\documentclass[]{article} 
\pdfoutput=1

\usepackage{amsfonts}
\usepackage{amsmath}
\usepackage{amssymb}
\usepackage{graphicx}
\usepackage{version}
\usepackage{epstopdf}
\usepackage{epsfig}
\usepackage{psfrag}
\usepackage{color}
\usepackage{setspace}
\usepackage{siunitx}
\usepackage{url}
\usepackage{amsthm}
\usepackage{authblk}

\newcounter{thm}
\newcounter{ex}
\newcounter{re}
\newtheorem{theorem}[thm]{Theorem}

\newtheorem{observation}[thm]{Observation}

\newtheorem{remark}[re]{Remark}
\theoremstyle{definition}
\newtheorem{example}[ex]{Example}

\newcommand{\R}{{\mathbb R}}
\newcommand{\Z}{{\mathbb Z}}

\newcommand{\waut}{\ensuremath{\mathit{Waut}}}

\newcommand{\stab}{\ensuremath{\mathit{stab}}}
\newcommand{\staba}{\ensuremath{\stab_\auta}}
\newcommand{\assembly}{\mathcal{A}}
\newcommand{\assemblycfg}{\ensuremath{\mathcal{B}}}
\newcommand{\auta}{\ensuremath{{\mathit{Waut}\!_\assembly}}}
\newcommand{\cfgspace}{\assembly} 
\newcommand{\strat}{\ensuremath{\mathcal{S}(\assembly)}}
\newcommand{\substrat}{\ensuremath{\tilde{\mathcal{S}}(\assembly)}}

\newcommand{\acg}{\ensuremath{G(\assemblycfg)}}

\def\sqr#1#2{{\vcenter{\vbox{\hrule height.#2pt
 \hbox{\vrule width.#2pt height#1pt \kern#1pt
 \vrule width.#2pt}
 \hrule height.#2pt}}}}

  \newcommand{\note}{\noindent \textbf{Note: }}
  \newcommand{\dfn}[1]{\emph{#1}}

\title{Symmetry in Sphere-based Assembly Configuration Spaces}

\author[1]{Meera Sitharam} 
\author[2]{Andrew Vince}
\author[1]{Menghan Wang}
\author[2]{Mikl\'os B\'ona}

\affil[1]{Department of Computer and Information Science and Technology, University of Florida}
\affil[2]{Department of Mathematics, University of Florida}

\begin{document}
	\maketitle
\begin{abstract}
Many remarkably robust, rapid and 
spontaneous self-assembly phenomena occurring in nature can be modeled 
geometrically, starting from a collection of rigid bunches of spheres.
 This paper highlights the role of symmetry 
in sphere-based assembly processes.
Since  spheres within bunches could be identical and bunches could be identical as well,
the underlying symmetry groups could be of large order that grows with the number of 
participating spheres and bunches.
Thus, understanding symmetries and associated isomorphism classes of microstates that correspond to various 
types of
macrostates 
can significantly increase efficiency and accuracy, i.e., reduce the  notorious 
complexity  of computing
entropy and free energy, as well as paths and kinetics, in high
dimensional configuration spaces.  In addition, a precise understanding of symmetries is crucial for giving provable guarantees of algorithmic accuracy and efficiency as well as accuracy vs. efficiency trade-offs in such computations.
In particular, this may aid in 
predicting crucial assembly-driving interactions.
 
This is a primarily expository paper that
develops a novel, original framework for dealing with symmetries in
configuration spaces of assembling spheres, with the following goals. 
(1) We give new, formal definitions of various concepts relevant to the sphere-based assembly setting that occur in previous work, and in turn,
formal definitions of their relevant symmetry groups leading to the main theorem concerning their symmetries. These previously developed concepts include, for example,
(a) assembly configuration spaces, (b) 
stratification of assembly configuration space into configurational regions defined by active constraint graphs, 
(c) paths through the configurational regions, and (d) coarse assembly pathways. 
(2) We then demonstrate the new symmetry concepts to 
compute sizes and numbers of orbits in two example settings appearing in previous work. (3) Finally, we  give formal statements of 
a variety of
open problems and challenges using the new conceptual definitions.

\end{abstract}

\section{Motivation}
\label{sec:1}
\label{sec:intro}

Supramolecular assembly is 
prevalent in nature, health-care 
and engineering, but poorly understood. 
The assembly starts with identical copies of structures drawn from 
a small number of types.
 Modeling these starting structures as rigid-bunches-of-spheres 
is well-suited to assembly processes driven by 
so-called short-range or hard sphere interaction potentials. 

More formally, an input to a computational model of an assembly process 
is an \emph{assembly system} 
consisting of the following.
\begin{itemize}
\item 
A collection of $k$
\emph{rigid molecular components} belonging to a few types;  
a rigid component is  specified as the set of 
positions of the centers of their constituent \emph{atoms}, in a local
coordinate system. In many cases, an \emph{atom} could be the representation
for the average position of a \emph{collection of atoms in an amino acid 
residue}.
Note that an assembly \emph{configuration} is 
given by the positions and orientations 
of the entire set of $k$ rigid molecular components in an assembly system, 
relative to one fixed component. Since each rigid molecular component has six
degrees of freedom, a configuration is a point in $6(k-1)$ dimensional
Euclidean space.
\item
The pairwise component of the potential energy function of the assembly system, specified as a sum of 
potential energy  terms between 
pairs of constituent atoms $i$ and $j$ in two different rigid components of
the assembly system. 
The weak interaction between the rigid molecular components 
is captured by this potential energy function.
The pairwise potential energy terms are, in
turn, specified using
\emph{pairwise potential energy functions} similar to so-called \emph{Lennard-Jones} potentials and Morse potentials~\cite{doye1997structural}. 
The potential energy is a function of the distance 
$d_{i,j}$ between $i$ and $j$.

\item
A non-pairwise component of the potential energy function in the form of \emph{global potential energy} 
terms that
capture the tethers between the rigid components within a monomer, as well as
other global potential energy terms that implicitly represent the solvent 
(water or lipid bilayer
membrane) effect \cite{Lazaridis_Karplus_1999, Lazaridis_2003,
Im_Feig_Brooks_2003}. These are independent of particular pairs of atoms.
\end{itemize}
It is important to note that all the above 
potential energy terms are \emph{functions 
of the assembly configuration}.


The formal conceptual framework we develop here
is inspired by  the following types of  prediction questions.
\begin{itemize}
\item
\emph{Input:} the 3D descriptions of the rigid molecular components
and their interactions (Section~\ref{sec:formal} 
describes how they are formally specified). \emph{Output:}  prediction of  
the final assembly structures and their
likelihood.
\item
\emph{Input:} as in the previous item, plus a 3D configuration of final
assembled structure. 
\emph{Output:}  prediction of those interactions that are crucial
for the assembly process to terminate in the given input assembly
configuration.
\item
\emph{Input:} as in the previous item.
\emph{Output:}  prediction of minimal alterations of the building blocks or 
interactions
that would significantly increase likelihood of the assembly process
terminating in the given input assembly configuration.
\item
\emph{Input:} as in the previous item, additionally more than one choice of
final assembly configuration.
\emph{Output:} prediction of key events such as specific intermediate
subassembly configuration choices during assembly that 
determine which one of the final assembly configuration is more 
likely to result.
\end{itemize}

Experimentally in vitro or vivo, these types of predictions about
supramolecular assembly processes are 
difficult because of the
remarkable rapidity, spontaneity and robustness of assembly processes.
The prediction tasks highlight combinatorial explosion and thus insufficiency of 
experimentation (trying various possibilities) and guesswork, 
even with the help of known data on
similar assemblies and biological knowledge about evolutionarily 
conserved structures. 
In addition, many of the current experimental methods 
are labor and
resource-intensive, making blind alleys expensive in time and effort.

On the other hand, 
computer simulations  
guided by theoretical first principles and
standard paradigms such as Monte Carlo(MC) or Molecular Dynamics(MD) 
are limited
due to the reasons detailed in the next subsections. 
\subsection{Assembly Configurational Volume}
\label{sec:volume}

Stability and binding affinity of  subassemblies 
depend on free energy whose landscape in the case of assembly is 
heavily influenced by configurational entropy (volume measure of microstates corresponding to a macrostate; see~\cite{kaku}); this
depends on accurate computation of configurational volumes by
sampling, attempted by a long and distinguished series of methods
\cite{kaku, Andricioaei_Karplus_2001,
    Hnizdo_Darian_Fedorowicz_Demchuk_Li_Singh_2007, 
    Hnizdo_Tan_Killian_Gilson_2008,
    Hensen_Lange_Grubmuller_2010,
    Killian_Yundenfreund_Kravitz_Gilson_2007,
    Head_Given_Gilson_1997,GregoryS201199,doi:10.1021/jp2068123}.
Assembly
configuration spaces are high dimensional,  and the number of required samples 
is typically exponential in the dimension.  Sampling on a high-dimensional ambient space 
grid typically means computing a large proportion
 of samples that lie outside any region of interest which is effectively of lower dimension,
and these samples must be discarded.
Not only are the relevant regions in the case of short-ranged potentials  of effectively lower dimension, 
they are also geometrically/topologically complex, 
hence grid-based sampling in Cartesian space, as well as nonergodic methods like MC or MD,
have to generate impractically dense sampling to accurately reflect the volume/measure ratios of these 
important, relatively low volume regions having complex geometry and topology.
These methods  do not exploit the abundance of symmetries of the landscape.
They are used both for \emph{assembly} processes, whose feasible regions are defined by one-sided pairwise distance equalities and inequalities between atom-centers, 
and \emph{folding} processes, where the feasible regions are defined by pairwise distance equalities.
The difference of complexity between the two is a litmus test
for the limitations that are addressed by the Cayley configuration space approach taken 
by EASAL described in Section~\ref{sec:easal}.

Conventional methods  to compute the energy landscape of small clusters are based on
searching for local minima~\cite{walesBook,doye1999double,doye1997structural}.
Point group symmetrisation schemes~\cite{wales0,wales1,wales2} and 
local rigidification schemes~\cite{wales3,wales4}
have been exploited in global optimisation algorithms
to gain computational efficiency.

Because of the complexity of the  problem of dealing with the short range
of interaction of hard spheres leading to narrow regions of lower
potential energy, separated by vast flat parts, conventional local-minima based methods  for energy landscape computation~\cite{walesBook} are
limited. These methods have the additional disadvantage  of small
perturbations to energy values requiring complete recomputation and also
they do not deal well with the very flat  landscape that is the signature
of short-range potentials.

An alternative approach for short-range potentials is to consider the ``sticky sphere limit'' based on taking the limit as the range of interaction goes to zero \cite{baxter1968,stell1991,frenkel2003}. In this limit, the energy landscape reduces to a collection of manifolds of different dimensions, glued together at their boundaries (formally, a Thom-Whitney stratification of real semi-algebraic sets), as described in theoretical models proposed independently and separately by Holmes-Cerfon et al.~\cite{Holmes-Cerfon2013} in 2013 and by the first author's research group~\cite{OzSi:2010,OPPS} in 2011.


The background provided in the remainder of this section
recalls previously developed concepts for describing assembly configuration spaces.
This motivates 
the conceptual framework 
for symmetry in assembly under short-ranged potentials
given in Section~\ref{sec:3}.

\subsection{Kinetics, Topology and Geometric complexity}

Kinetics and 
transition rates between subassemblies also 
require an explicit understanding of the geometry, topology 
and multiple paths in the assembly configuration space. 
For cluster assemblies from spheres,
there are a number of methods 
\cite{Arkus2009,Wales2010,Beltran-Villegas2011,Calvo2012,Khan2012,Hoy2012,Hoy2014}
 to compute the entire configuration space of small molecules such as cyclo-octane
\cite{Martin2010,Jaillet2011,Porta2007}.  
Some methods from 
robotics and computational geometry~\cite{GregoryS201199},
such as the probabilistic roadmap~\cite{Amato2002}, 
 effectively give bounds to approximate
free energy without relying on MC\ or MD\ sampling.  
Starting from MC\ and MD\ samples, recent heuristic methods 
infer topology
\cite{Gfeller_DeLachapelle_DeLos_Rios_Caldarelli_Rao_2007,%
Varadhan_Kim_Krishnan_Manocha_2006,Lai_Su_Chen_Wang_2009,%
10.1371/journal.pcbi.1000415}, and use 
topology to guide dimensionality reduction 
\cite{Yao_Sun_Huang_Bowman_Singh_Lesnick_Guibas_Pande_Carlsson_2009}. 
 Yet, most prevailing methods are unable to  extract the topology in a sufficiently
efficient and accurate manner as to be able to feasibly compute volume or path integrals (required for
entropy or kinetics computations)
even for small assemblies. 
Moreover even those prevailing methods that exploit symmetry in the configuration space
to compute free energy and kinetics do not employ a formal and precise group-theoretic framework.

\subsection{Recursive decomposition, Assembly trees, Combinatorial entropy}

For larger, microscale assemblies, a direct study of the free
energy and configurational entropy is computationally emphatically
intractable.  At these coarser scales, the primitives  are stable subassemblies and  
transition rates (obtained from the 
computational tasks of the previous two subsections). Still, the combinatorial entropy of
multiple pathways makes
it difficult to isolate \emph{crucial 
combinations of assembly-driving interface interactions}.

This issue has been addressed by the first author's previous work on recursive decompositions \cite{bib:HoLoSi99,bib:HoLoSi98b,bib:HoLoSi98c}
of larger assemblies into smaller subassemblies. 
This work introduces structures
called assembly trees and the notion of combinatorial entropy, applied to model viral capsid assembly in \cite{AgSi:06}. 

While trees of various types have been used to model various processes related to assembly \cite{carvalho2010stepwise,wales2006energy}, to the best of our knowledge, the assembly trees from \cite{AgSi:06} have a formal structure that is distinct from other tree representations of assembly pathways. In particular,  non-root nodes of the assembly tree contain subassemblies, rather than configurations of the entire assembly system; and any pair of nodes that are incomparable (neither ancestor or child in the tree) are disjoint subassemblies, i.e, they do not contain any common rigid components; moreover, only rigid subassembly configurations are represented.   
In addition, the authors have 
taken the first steps towards precisely formalizing the effect of symmetries 
on  a highly simplified version of assembly trees; specifically their orbits under the 
action of a fixed group of symmetries -- called assembly pathways \cite{BSV}. These concepts will be discussed in detail in Sections \ref{sec:3} and \ref{sec:4}.
\subsection{Symmetry in Chemistry}

Since  spheres within rigid bunches of an assembly system could be identical and bunches could be identical as well,
the underlying symmetry groups could be of large order, that grows with the number of 
participating spheres and bunches.
Therefore, 
all of the tasks in the previous three subsections can be significantly 
simplified by taking advantage of natural symmetries of the configuration space
that arise due to identical assembling units, their symmetries, and symmetries 
of the final assembled structure. 
However, none of the 
prevailing methods discussed above
computationally incorporates these symmetries.
Group theory has been used to study the symmetry of molecules and molecular orbits
 ~\cite{bunker2004fundamentals,cotton2008chemical,bonchev1995chemical,
kerber2013mathematical} for a long time.
The well-known P\'olya enumeration theorem~\cite{polya1987combinatorial},
which provides a method to find the number of orbits of a group action,
is motivated by the problem of enumerating permutational isomers of a given 
molecular skeleton.
Group theory is widely used in crystallography
to describe crystallographic symmetry and classify crystal structures~\cite{altmann1977induced,hahn2005international}.
Other applications include 
using the molecule symmetry group in studying molecular 
spectroscopy~\cite{bunker1998molecular}
and using generating functions in understanding nuclear spin statistics of 
nonrigid molecules  \cite{balasubramanian1981generating}.
However, most of these works only involve symmetry of individual structures.
The literature is sparse in the context of symmetry in assembly 
systems or in configuration spaces.
\subsection{EASAL: Efficient Atlasing and Search of Assembly Landscapes}
\label{sec:easal}

A recent method of the first author, EASAL   
({\em efficient atlasing and search of assembly landscapes}) 
\cite{OzSi:2010, OPPS}, 
formally addresses the issues highlighted in the first two subsections above:
computation of configurational entropy and kinetics,  via {\em geometrization,}  
{\em stratification} and {\em convexification using Cayley parameterization} 
of assembly configuration spaces.  
Geometrization and Stratification were also used later in \cite{Holmes-Cerfon2013} independently (as mentioned at the end of Section~\ref{sec:volume}):  the geometrization is achieved in \cite{Holmes-Cerfon2013} via a somewhat different process consistent with smooth potential energy functions, while the stratification  is the standard Thom-Whitney stratification of semialgebraic sets as laid out in \cite{OzSi:2010, OPPS}.

On the other hand, Cayley convexification based on \cite{sitharam2010convex} is a unique feature of EASAL not present in \cite{Holmes-Cerfon2013},
that makes it tractable to sample and compute entropy integrals over
higher dimensional  constant-potential-energy regions of the assembly
configuration space. In addition Cayley convexification helps formalize
and precisely explain the intuitively clear observation that \emph{assembly}
configuration spaces are significantly simpler geometrically and
topologically than  \emph{folding} configuration spaces.
The difference in complexity is especially stark when there are cycles of pairwise constraints between atom centers.

We describe the {\em Geometrization} and {\em Stratification} aspects  of EASAL's approach below.
Stratification is
explained in further
detail in Section \ref{sec:3} and  Cayley parameters for configuration spaces and convexification based on~\cite{sitharam2010convex} are explained in Section \ref{sec:5}.
\subsubsection{Geometrization}

The assembly configuration space is represented as a semi-algebraic set satisfying geometric constraints specified
as distance inequalities between atom-centers.
The short-range  or hard sphere potential interaction is 
typically 
discretized to take different constant values on 
three intervals for the distance value $d_{i,j}$:
$(0,r_{i,j})$, 
$(r_{i,j}, r_{i,j}+\delta_{i,j})$, and
$(r_{i,j}+\delta_{i,j}, \infty).$  Typically, $r_{i,j}$, the so-called Van der
Waal or steric radius,
specifies "forbidden" regions around atoms $i$ and $j.$ And $r_{i,j}+\delta_{i,j}$ is a
distance where the attractive (electrostatic or other weak) forces 
between the two atoms is no longer strong (typically these forces decay
as the reciprocal of some power of the distance $d_{i,j}$ between atom centers). 
Intuitively, the interval 
$(0,r_{i,j})$ is where the repulsive force highly dominates, and  
$(r_{i,j}, r_{i,j}+\delta_{i,j})$ is where the attractive force and repulsive forces are
balanced, and 
$(r_{i,j}+\delta_{i,j}, \infty)$ is where neither force is strong.  
Over these 3 intervals respectively, the potential assumes a 
very high value, a very low value, and a medium value
$m_{i,j}.$  All of these \emph{bounds} for the intervals for $d_{i,j}$, as well as the values for the 
 potential on these intervals, 
are specified as part of the
input to the assembly model. These constants are specified for 
each pair of atoms $i$ and $j$,
i.e., the subscripts are necessary.
The interval with the low value is called the \emph{well}.
The Hard-Sphere potentials are defined solely by the Van der Waal's forbidden
distance constraint, $\delta_{i,j} = 0$.

The information in the potential energy landscape can thus be 
geometrized, i.e., represented using
\emph{assembly constraints}, in the form of distance intervals.
These constraints define \emph{feasible} configurations. The set of feasible
configurations is called the \emph{assembly configuration space}.
The \emph{active constraint} regions of the
configuration space are regions where at least one of the short-ranged 
inter-atom distances lies in the potential energy well, i.e, the 
interval $(r_{i,j}, r_{i,j} +\delta_{i,j})$.
\subsubsection{Stratification}

The above geometrization of an assembly configuration space makes it natural to stratify an assembly
configuration space into {\sl atlas of active constraint regions}, 
More details are provided in Section \ref{sec:3} -- see also Figure~\ref{fig:roadmap}.
The {\it active constraint regions} of the
configuration space are regions where at least one of the 
inter-atom distances lies in the potential energy well.
The active constraint regions are
stratified by dimension into a topological Thom-Whitney complex,
with the boundary  region being  one dimension
smaller.
The active constraint regions 
can be modeled as so-called convexifiable
Cayley configuration spaces \cite{sitharam2010convex}, a combinatorially definable concept by first labeling each region
by its unique active
constraint graph (see Section \ref{sec:3}).
A demo movie of EASAL is available at:
\url{http://www.cise.ufl.edu/research/SurfLab/EASALvideo.mpg}.
Standard algorithms can be employed for a fast 
computation of paths from one configuration to another in the atlas. 
However, the computation of entropy integrals over these paths poses several
challenges.
%
 %
%
 
%
\subsection{Organization and Contribution}

This is a primarily expository paper that
develops a novel, original framework for dealing with symmetries in
configuration spaces of assembling spheres under short ranged potentials. 
It is motivated by a longer-term goal to exploit natural symmetries 
 using assembly trees and other concepts described 
in the previous sections, that have appeared in various avatars in the community, including our
work on EASAL.
Such an understanding of symmetries is essential for significantly reducing the complexity of the computation of configurational and combinatorial entropy as well as kinetics,
since spheres within rigid bunches of an assembly system could be identical and bunches could be identical as well,
giving underlying symmetry groups of large order, that grows with the number of 
participating spheres and bunches.

To this end, we develop a formal conceptual framework for assembly under short-ranged
potentials, as assembly  of rigid bunches of spheres.  
As different definitions of assembly macrostates are appropriate in different
contexts, for example, depending on whether different copies of identical
atoms or molecules  are considered interchangeable or not,
we carefully define and differentiate between congruence and isomorphism
of configurations.
We then show how symmetries 
of assembly configuration spaces 
arise due to: multiple copies of identical building blocks
(in particular when these building blocks are rigid bunches of spheres),
internal 
symmetries of building blocks, 
and the symmetries of the final assembled structure.

The organization of this paper is as following.
In Section~\ref{sec:3}, we define the new conceptual framework 
for symmetry in assembly under short-ranged potentials (or assembly of {\em
	rigid bunches of spheres}) leading to the main Theorem~\ref{thm:containment}.
An application of some of these results on symmetry can be found in \cite{OPPS}.
In Section~\ref{sec:4},
we illustrate one aspect of our approach~\cite{BSV}
for computing combinatorial entropy using
generating functions
for counting the number and size of simplified assembly pathways (orbits of a
symmetry group action on assembly trees).
Note that while this simple example has a fixed group size,
the method demonstrated applies also when the underlying symmetry group grows with the size of the system. 
In Section~\ref{sec:5}, open questions and directions are given.


\section{Framework for Symmetry in Assembly}
\label{sec:3}
\label{sec:formal}

In this section, we define natural groups of symmetries acting on
various previously defined objects related to symmetry that are described in the Introduction and later in this section. 
The four new groups we defined are the \emph{weak automorphism group}, the \emph{strict congruence group}, the \emph{strict order preserving isomorphism group} and the \emph{strict permuted congruence} group of an assembly configuration. 
We consider the action of these groups on various objects defined in previous literature on assembly and sketched in the introduction~\cite{OzSi:2010,OPPS,mvs2006}, 
such as assembly configuration space, active constraint regions, active constraint graphs, assembly paths and trees. 
 These resulting symmetry classes will be used to formalize the main new Theorem~\ref{thm:containment} and two applications in Example~\ref{eg:tree} and Section~\ref{sec:4}, as well as open problems in the last section of this paper.


Let $X$ be a set under the action of a group $G$, 
and $x$ be any element of $X$.
The \emph{orbit} of $x$ under $G$ is the set
$G(x) = \{ \phi(x): \phi \in G \}$.
An element $g$ of $G$ \emph{fixes} $x$ if $g(x) = x$.
The \emph{stabilizer subgroup} $\stab_{G}(x)$ of $x$ in $G$ is the group 
of all elements in $G$ that fix $x$,
 i.e.
\ $\stab_{G}(x)=\{\phi \in G| \phi(x) = x\}$.  


The following theorem from standard group theory 
can be used to determine the number of orbits and the size of orbits
for various objects defined in this section. An explicit
application of this theorem is
shown in the next section.

\begin{theorem}
Let $X$ be a set 
under the action of a group $G$.
%
%
For all $x \in X$, the equalities

\[|G(x)| = |G| / |\stab_G(x)| \qquad 
\text{(Orbit-Stabilizer theorem)}\]

and

\[|X / G| =  \frac{1}{|G|}\sum_{\phi \in 
G}|X^\phi| \qquad \text{(Burnside's lemma)}\]

hold,
where $|X / G|$ is the number of orbits of $X$, $X^\phi$ 
is the set $\{x \in X : \phi(x) = (x)\}$.

\end{theorem}

Different definitions of macrostates are appropriate in different
contexts, for example, depending on whether different copies of identical
atoms or molecules  are considered interchangeable or not. For this reason
we carefully define and differentiate between congruence and isomorphism
of configurations.

In order to give a physically meaningful formalization of an assembly system under short-ranged 
potentials, we  define the notion of a {\em bunch}, i.e., a
rigid configuration of spheres of varying colors and radii.

\subsection{A Bunch and its symmetries}

Let $SE(3)$ denote the group of orientation preserving isometries of $\R^3$.

  A \emph{bunch} is a tuple $(P; \mathcal C, r, \delta)$ where $P = (p_1, p_2, \ldots, p_n)$ is an ordered set of points in $\R^3$,
   and $\mathcal C, r, \delta$ are functions defining colored spheres centered at the points in $P$.
Specifically, $\mathcal C : P \rightarrow  C$ where $C$ is a finite set of ``colors", and $r, \delta : P \rightarrow \R^+$
such that the spheres are nonintersecting, i.e.\ $\|p_i - p_j\|_2 \ge r(p_i) + r(p_j)$ for any $i \ne j$.
The map $\delta$ is interpreted as the width of the annulus specified by the
potential energy well and is used in the definition of an active
constraint graph of an assembly configuration later in this section. 
For a bunch $B$, $P(B)$ is used to denote the point set $B$; similarly we have $\mathcal{C}(B), r(B)$ and $\delta(B)$.

Two bunches $B= (P; \mathcal C, r, \delta)$ and $B' = (P'; {\mathcal C}', r', \delta ')$
are \emph{isomorphic} 
if there is an  element $\phi$ of $SE(3)$ and a permutation $\pi \in S_n$,
 such that
$\phi(p_i) = p_{\pi(i)}'$ for all $i$, where $n = |P|$,
and $\phi$ preserves the color, radius and annulus of points.
%
%
In this case with a slight abuse of notation, 
we write $B' \in \phi(B)$,
where $\phi(B)$ denotes the set of bunches that are isomorphic to $B$ under $\phi$ and some permutation in $S_n$.
See Figure~\ref{fig:bunch_isomorphic} for an example.

Two bunches $B= (P; \mathcal C, r, \delta)$ and $B' = (P'; {\mathcal C}', r', \delta ')$
are \emph{strictly isomorphic}, if there is a  permutation $\pi \in S_n$
such that $B$ and $B'$ are isomorphic under $\pi$ and the identity element in $SE(3)$.
The \emph{weak automorphism group} of $B$, denoted $\waut(B)$, is the group of all permutations $\pi \in S_n$ 
that take $B$ to a strictly isomorphic $B'$.


\begin{figure}[hbtp]
	
\centering
\includegraphics[width=0.4\linewidth]{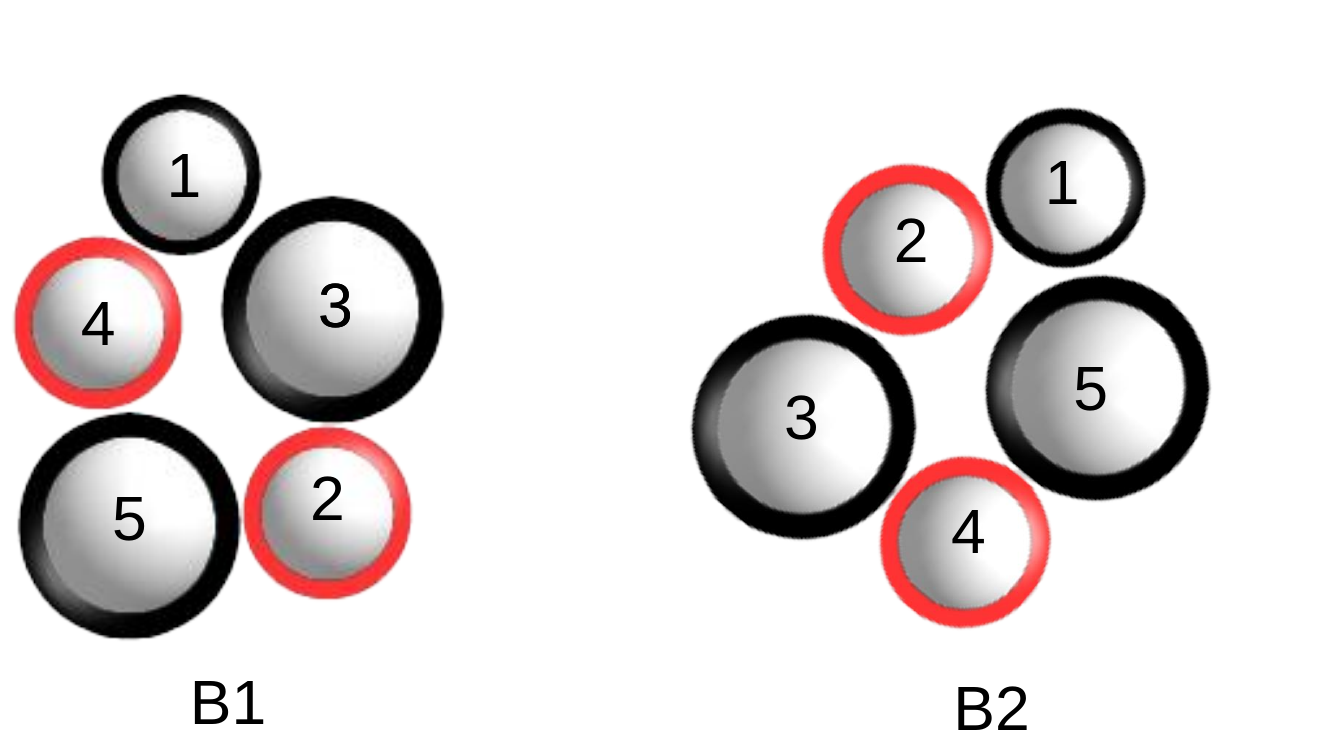}

\caption{Two isomorphic bunches of 5 spheres}

\label{fig:bunch_isomorphic}

\end{figure}

Two bunches $B= (P; \mathcal C, r, \delta)$ and $B' = (P'; {\mathcal C}', r', \delta ')$
are \emph{order preserving isomorphic} or \emph{congruent}, if there is a $\phi \in SE(3)$
such that
$B$ and $B'$ are isomorphic under $\phi$ and the identity permutation.
In this case with a slight abuse of notation, 
we write $B' = \phi(B)$.

We have the following observation that describes strict isomorphism using the notion of congruence.

\begin{observation}
Two congruent bunches $B$ and $B'$ are strictly isomorphic, if and only if 
$\tilde{P} = \tilde{P}'$ where $\tilde{P}$ and $\tilde{P}'$ denote the unordered point sets of $B$ and $B'$ respectively, 
and for all $p \in P'$, $\mathcal{C}'(p) = \mathcal{C}(p)$, $r'(p) = r(p)$, $\delta'(p) = \delta(p)$.
\end{observation}

\subsection{An assembly configuration space and its symmetries}

An \emph{assembly configuration} is an ordered set $\mathcal{B} = (B_1, B_2 \ldots B_k)$ where $B_i = (P_i; \mathcal{C}_i, r_i, \delta_i)$ is a bunch for all $i$,
such that for all $i,j$ and all $x \in P_i, \; y \in P_j, \; x\neq y$,
we have
\begin{equation} \label{eq:ac} \|x - y\|_2 \ge r_i(x) + r_j(y) \end{equation}

Two assembly configurations $\mathcal B = (B_1,\ldots, B_k)$
and $\mathcal B' = (B_1', \ldots, B_k')$  
are configurations of the same
\emph{assembly system} (see Section~\ref{sec:1})
if $B_i$ is congruent to $B_{\sigma(i)}'$ for some permutation $\sigma \in S_k$, for all $i$. 
Notice that the congruence between bunches could be different for each $i$. 
The set of all assembly configurations of an assembly system is called an \emph{assembly configuration space}.
The assembly configuration space containing the assembly configuration $\assemblycfg$ is denoted $\assembly(\assemblycfg)$, 
or simply $\assembly$ when the context is clear.

In the following discussion, we always restrict our universe to assembly configurations in the same assembly configuration space.


Two assembly configurations $\assemblycfg = (B_1, \ldots, B_k)$ and $\assemblycfg' = (B_1', \ldots, B_k')$ 
are \emph{isomorphic} if
there is an  element $\phi$ of $SE(3)$ (isomorphism between bunches) and a permutation $\sigma \in S_k$,
 such that for all $i$, $B_{\sigma(i)}'$ is isomorphic to $B_i$ under $\phi$ and a permutation $\pi_i \in S_{n_i}$, where $n_i = |P_i|$.


Two assembly configurations $\assemblycfg$ and $\assemblycfg'$
are \emph{strictly isomorphic}, if there is a  permutation $\sigma \in S_k$,
 such that for all $i$, $B_{\sigma(i)}'$ is isomorphic to $B_i$ under the identity element in $SE(3)$ and a permutation $\pi_i \in S_{n_i}$, where $n_i = |P_i|$.
 Thus a strict isomorphism is a tuple of permutations $(\sigma, \pi_1, \ldots, \pi_k)$, where $\sigma \in S_k$ and $\pi_i \in S_{n_i}$.
The \emph{weak automorphism group} of $\assemblycfg$, denoted $\waut(\assemblycfg)$, is the group of all such tuples $(\sigma, \pi_1, \ldots, \pi_k)$ 
that take $\assemblycfg$ to a strictly isomorphic $\assemblycfg'$, 
with the group operation $(\sigma, \pi_1, \ldots, \pi_k)(\sigma', \pi'_1, \ldots, \pi'_k) = (\sigma\sigma', \pi_1\pi'_1, \ldots, \pi_k\pi'_k)$.

Note that all assembly configurations in the same assembly configuration space $\assembly$ have the same weak automorphism group.
Thus we define the \emph{weak automorphism group} of an assembly configuration space $\assembly$, denoted $\auta$, to be the weak automorphism group of any assembly configuration $\assemblycfg$ in $\assembly$.

Two assembly configurations $\assemblycfg$ and $\assemblycfg'$
are \emph{congruent} if there is an isomorphism $\phi \in SE(3)$ that preserves both the order of the bunches and the order of points within each bunch, 
i.e.\ for all $i$, $B_i'$ is congruent to $B_i$ under $\phi$.
Two assembly configurations $\assemblycfg$ and $\assemblycfg'$
are \emph{strictly congruent} if they are both congruent and strictly isomorphic.
In general, we think of two strict congruent assembly configurations as the same.
The \emph{strict congruence group} of an assembly configuration $\assemblycfg$ 
is the stabilizer of the set strictly congruent assembly configurations of $\assemblycfg$
under $\auta$.
It is the stabilizer subgroup $\stab_{\auta}\assemblycfg$ of the assembly configuration $\assemblycfg$ under $\auta$.

Two assembly configurations $\assemblycfg$ and $\assemblycfg'$
are \emph{order preserving isomorphic} if there is an isomorphism $\phi \in SE(3)$ that preserves the order of the bunches, 
i.e.\ for all $i$, $B_i'$ is congruent to $\phi(B_i)$.
Two assembly configurations $\assemblycfg$ and $\assemblycfg'$
are \emph{strictly order preserving isomorphic} if they are both order preserving isomorphic and strictly isomorphic.
The \emph{strict order preserving isomorphism group} of an assembly configuration $\assemblycfg$
is the stabilizer of the set of strictly order preserving isomorphic configurations of $\assemblycfg$
under $\auta$.

Two assembly configurations $\assemblycfg$ and $\assemblycfg'$
are \emph{permuted-congruent} if there is an isomorphism that preserves the order of points within each bunch, 
i.e.\ 
there is an  element $\phi$ of $SE(3)$ and a permutation $\sigma \in S_k$,
such that for all $i$, $B'_{\sigma(i)}$ is congruent to $B_i$ under $\phi$.
Two assembly configurations $\assemblycfg$ and $\assemblycfg'$
are \emph{strictly permuted-congruent} if they are both permuted-congruent and strictly isomorphic.
The \emph{strict permuted congruence group} of an assembly configuration $\assemblycfg$
is the stabilizer of the set of permuted-congruent configurations of $\assemblycfg$
under $\auta$.

As an example, refer to Figure~\ref{fig:weak_automorphism}.
The assembly configuration $\assemblycfg_1$ consists of 3 congruent bunches.
The assembly configuration $\assemblycfg_2$ is obtained from $\assemblycfg_1$ with a strict congruence  $(\sigma, \pi_1, \pi_2, \pi_3)$ induced by a rotation in $SE(3)$,
where  $\sigma = (1 \; 3)$, and $\pi_i = id$ for all $i$.
The assembly configuration
$\assemblycfg_3$ is obtained from $\assemblycfg_1$ with a strict permuted congruence $(\sigma, \pi_1, \pi_2, \pi_3)$,
where  $\sigma$ is a cyclic permutation of the 3 bunches, and $\pi_i = id$ for all $i$.
On the other hand,
 $\assemblycfg_4$ is obtained from $\assemblycfg_1$ with a strict isomorphism $(\sigma, \pi_1, \pi_2, \pi_3)$,
where  $\sigma$ is a cyclic permutation of the 3 bunches, $\pi_1 = (1\;2)$ and $\pi_2 = \pi_3 = id$.

\begin{figure}[hbtp]
\centering
\includegraphics[width=0.7\linewidth]{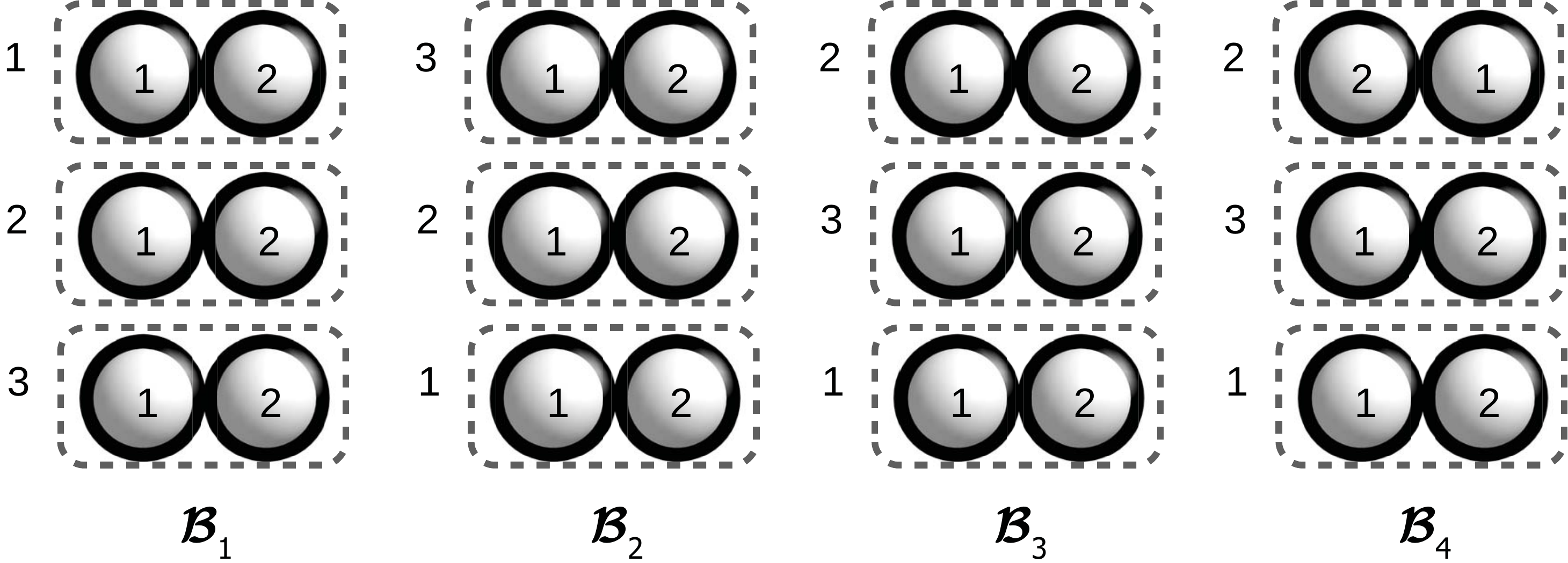}
\caption{The assembly configuration $\assemblycfg_1$ consists of 3 isomorphic bunches.
$\assemblycfg_2$ is obtained from $\assemblycfg_1$ with a strict congruence,  $\assemblycfg_3$ is obtained from $\assemblycfg_1$ with a strict permuted congruence, and $\assemblycfg_4$ is obtained from $\assemblycfg_1$ with a strict isomorphism that is neither a strict congruence, nor a strict permuted congruence, nor a strict order preserving isomorphism.
}
\label{fig:weak_automorphism}
\end{figure}

Figure~\ref{fig:assembly_config_orbit} shows another example
of four  assembly configurations each containing two bunches.
The strict congruence group $\stab_{\auta}\assemblycfg$ of the assembly configuration $\assemblycfg_1$ is of size~2 and contains those tuples $(\sigma, \pi_1, \pi_2)$, where  $\pi_1 \in \{id, (2\;4)\}$, $\sigma = id$, $\pi_2 = id$. 
The weak automorphism group $\auta$ of the assembly system is of size~4 and contains those tuples $(\sigma, \pi_1, \pi_2)$, 
where $\pi_1 \in \{id, (2\;4),(3\;1), (2\;4)(3\;1)\}$, $\sigma = id$, $\pi_2 = id$.
All four strictly isomorphic assembly configurations are obtained by applying
$\auta$ to the assembly configuration $\assemblycfg_1$. 
Notice that $\assemblycfg_2$ and $\assemblycfg_1$ ($\assemblycfg_4$ and $\assemblycfg_3$) are strictly congruent,
while $\assemblycfg_3$ and $\assemblycfg_1$ are strictly order preserving isomorphic.
The orbit of $\assemblycfg_1$ under $\auta$ is of size 2 and consists of $\assemblycfg_1$ and $\assemblycfg_3$. 

\begin{figure}[hbtp]
\centering
\includegraphics[width=0.65\linewidth]{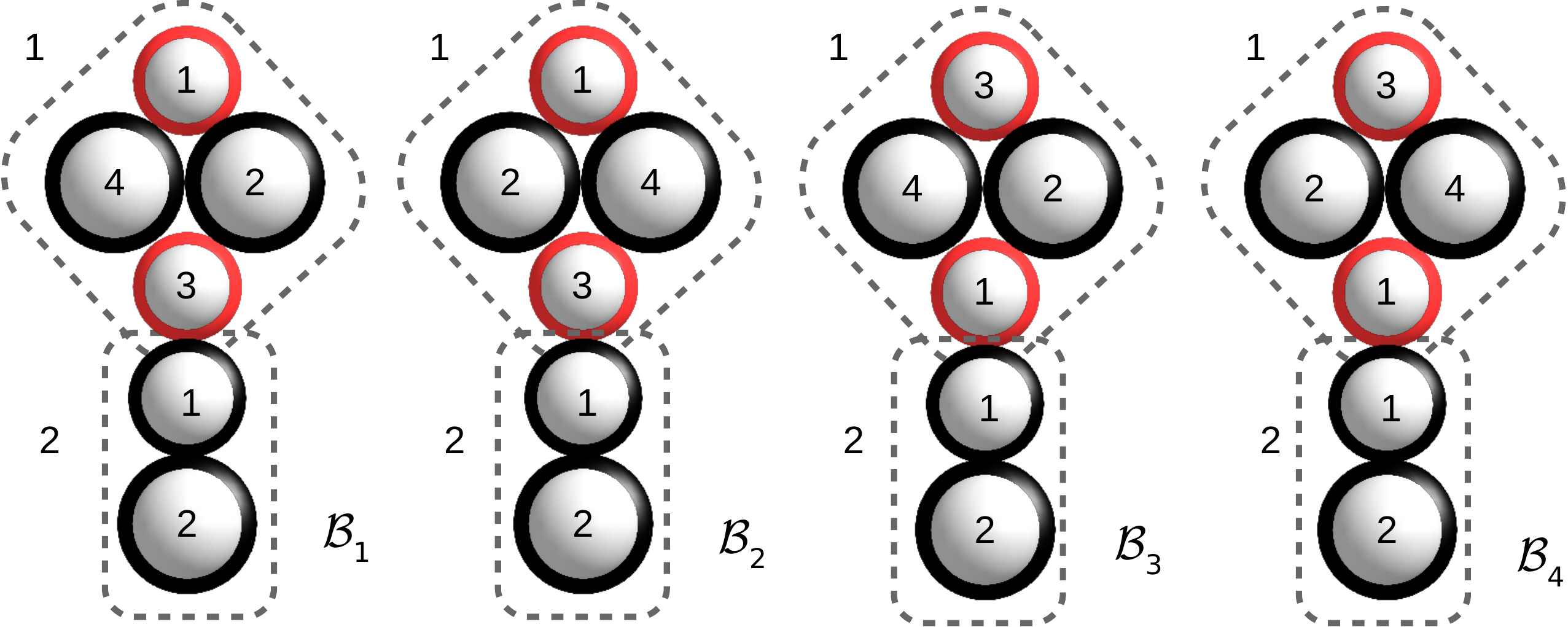}
\caption{Four assembly configurations obtained by applying $\auta$ on the assembly configuration $\assemblycfg_1$. 
$\assemblycfg_2$ is obtained from $\assemblycfg_1$ with a congruence, while $\assemblycfg_3$ is obtained from $\assemblycfg_1$ with a strict order preserving isomorphism. 
}
\label{fig:assembly_config_orbit}
\end{figure}


%

We have the following observations for alternative characterizations of strict congruence, strict order preserving isomorphism and strict permuted congruence of assembly configurations.
\begin{observation}
Given two assembly configurations $\assemblycfg = (B_1, \ldots, B_k)$ and $\assemblycfg' = (B_1', \ldots, B_k')$ in the same assembly configuration space,
\begin{enumerate}
\item $\assemblycfg$ and $\assemblycfg'$ are strictly congruent
 if and only if they are congruent, and
 \begin{itemize}
 \item[(*)] $\assemblycfg$ and $\assemblycfg'$ have the same unordered partition of the unordered point set into bunches, 
 i.e.\
$ \{\tilde{P_1}, \ldots, \tilde{P_k}\} = \{\tilde{P'_1}, \ldots, \tilde{P'_k}\}$, where $\tilde{P_i}$ is the unordered point set of the bunch $B_i$, 
 and each point has same color, radius and annulus in $\assemblycfg$ and $\assemblycfg'$.
 \end{itemize}
 
 \item $\assemblycfg$ and $\assemblycfg'$ are strictly order preserving isomorphic if and only if they are order preserving isomorphic and satisfy the condition (*).
 
  \item $\assemblycfg$ and $\assemblycfg'$ are strictly permuted congruent if and only if they are permuted congruent and satisfy the condition~(*).
\end{enumerate}
\end{observation}

\subsection{Symmetries in active constraint graph and active constraint region}

An \emph{active constraint graph} $G(\assemblycfg)$ of an assembly 
configuration $\assemblycfg = (B_1, \ldots, B_k)$ is a graph $(V,E)$, 
where the  vertex set $V$ has one vertex for each point $p \in P_1 \cup \ldots \cup P_k$, labeled by a tuple $(i, l)$,
 representing that the point $p$ appears as the $i^{th}$ point $p_i$ in the
$l^{th}$ bunch $B_l$ of $\assemblycfg$,
and a vertex pair $\{x,y\} \in E$  if $x$ and $y$ lie in distinct bunches of $\assemblycfg$ and 
\[r(x) + r(y) \leq \|x-y\|_2 \leq  (r(x) + \delta(x)) + (r(y) + \delta(y)).\] 



An element $(\sigma, \pi_1, \ldots, \pi_k)$ of the weak automorphism group $\auta$ of $\assemblycfg$'s assembly configuration space $\assembly$ 
acts on $G(\assemblycfg)$ by taking the tuple $(i,l)$ to $(\pi_l(i), \sigma(l))$.

Two active constraint graphs $G_1,G_2$ are \emph{isomorphic} 
if  there is a $\psi = (\sigma, \pi_1, \ldots, \pi_k) \in \auta$ such that 
$\{x, y\}  \in E(G_1) \Longleftrightarrow \{\psi(x), \psi(y)\} \in E(G_2)$. 
In this case we say $G_1 \cong_\psi G_2$ or $\psi(G_1) = G_2$.

The \emph{automorphism group} of an active constraint graph $G$ is the 
group of elements $\psi \in \auta$ such that $\psi(G) = G$, 
i.e.\ it is the stabilizer subgroup $\staba G$.

For example, 
%
Figure~\ref{fig:acg2} shows all the non-isomorphic 
active constraint graphs with 12 edges of an assembly system 
consisting of 6 bunches, where all bunches are identical singleton spheres.

%
%
%
%

\begin{figure}[hbtp]
\centering

\includegraphics[width=0.9\linewidth]{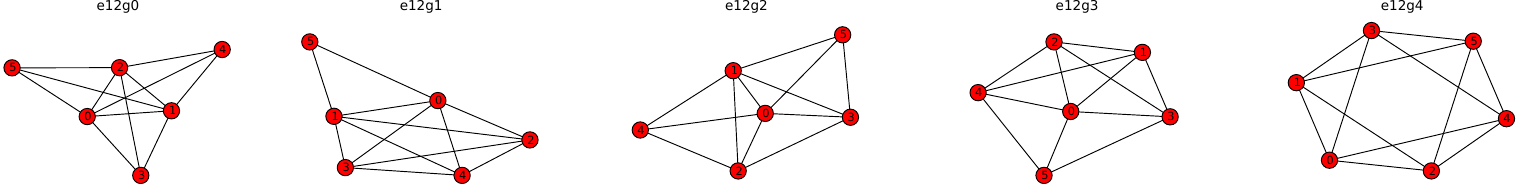}

\caption{All non-isomorphic active constraint graphs with 12 edges
of an 
assembly system of 6 bunches that are identical singleton spheres.
The label on top is automatically generated by EASAL and specifies the orbit number of the shown active constraint graph.}

\label{fig:acg2}

\end{figure}



%

\note It is clear that  $\staba{\assemblycfg} \subseteq \staba\acg$.
Moreover, there are assembly configurations $\assemblycfg$ such that 
$\staba{\assemblycfg} \subsetneq \staba\acg$, 
i.e.\ the strict congruence group of $\assemblycfg$ does not have all the automorphisms of the 
corresponding active constraint graph. Refer to the assembly 
configuration $\assemblycfg$ and its active constraint graph $G$ 
in Figure~\ref{fig:config_graph_aut}, where each bunch is a singleton sphere. 
The permutation $\sigma = (1\;2\;3) \in \auta$ is contained in $\staba(G)$.
However, it is not contained in the strict congruence group 
$\staba{\assemblycfg}$ of the assembly configuration.
\smallskip


\begin{figure}[hbtp]

\centering

\includegraphics[width=0.35\linewidth]{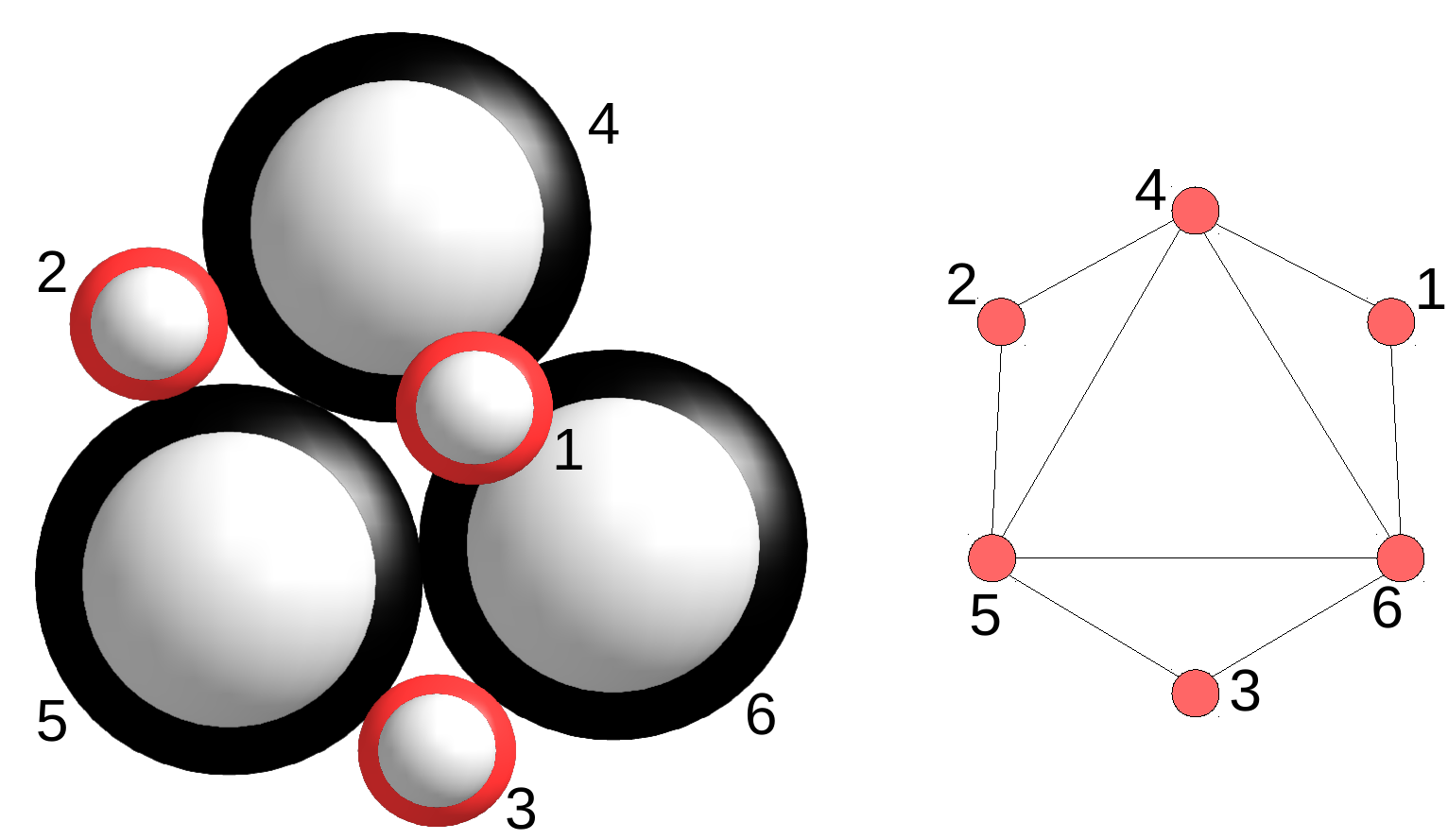}

\caption{An assembly configuration whose automorphism group is 
strictly contained in that of the corresponding active constraint graph.
Here the bunches are singleton spheres and bunches of the same color have the same $\mathcal{C}$, $r$ and $\delta$.}

\label{fig:config_graph_aut}

\end{figure}

The \emph{full graph $G^*$} of an active constraint graph $G$ is obtained 
by adding edges to $G$ to make the set of vertices in each bunch 
into a clique.


An \emph{active constraint region $R_G$} 
of the assembly configuration space $\cfgspace$  
contains all assembly configurations 
$\assemblycfg$ with the active constraint graph $\acg = G$.
The  action of elements of $\auta$ on an active constraint region,
and the stabilizer of an active constraint region in $\auta$ are
well-defined by the action of $\auta$ on assembly configurations.
%

%


%

The following theorem gives 
containment and equality relations between stabilizer subgroups of an active constraint graph, an active constraint region and individual configurations in the active constraint region. 
 
\begin{theorem}
	\label{thm:containment}
 
 For an active constraint graph $G = \acg$ of an assembly configuration space $\assembly$, it holds that
 \[
 \staba{\assemblycfg} \subseteq \staba G = \staba R_{G}
 \]
 In addition, there exist active constraint graphs $G$  of assembly configuration spaces $\assembly$ where the above containment is strict, i.e.\
 \[
  \text{ for all }\assemblycfg \text{ such that } G = \acg, \quad
\staba{\assemblycfg} \subsetneq \staba G = \staba R_{G}
 \]

 \end{theorem}
 
 \begin{proof}
 (1) It is straightforward to see that $\staba{\assemblycfg} \subseteq \staba \acg$.
 We give an example to show the existence of $G$ where $\staba{\assemblycfg} \subsetneq \staba G$ for any  assembly configuration $\assemblycfg$ of $G$.
 Refer to the assembly configuration in Figure~\ref{fig:config_graph_aut2}, where each bunch is a singleton sphere. 
 The permutation $\sigma = (1\; 2\; 3)$ is contained in the automorphism group $\staba G$ of the active constraint graph $G$.
 However, it is not contained in the strict congruence group of any corresponding assembly configuration, 
 as the position of the sphere 6 is asymmetric with respect to $1,2,3$ in any assembly configuration of $G$.
 Thus $\staba{\assemblycfg} \subsetneq \staba G$ for any assembly configuration $\assemblycfg$ of $G$.. 
 
 \begin{figure}[hbtp]
 \centering
 \includegraphics[width=0.35\linewidth]{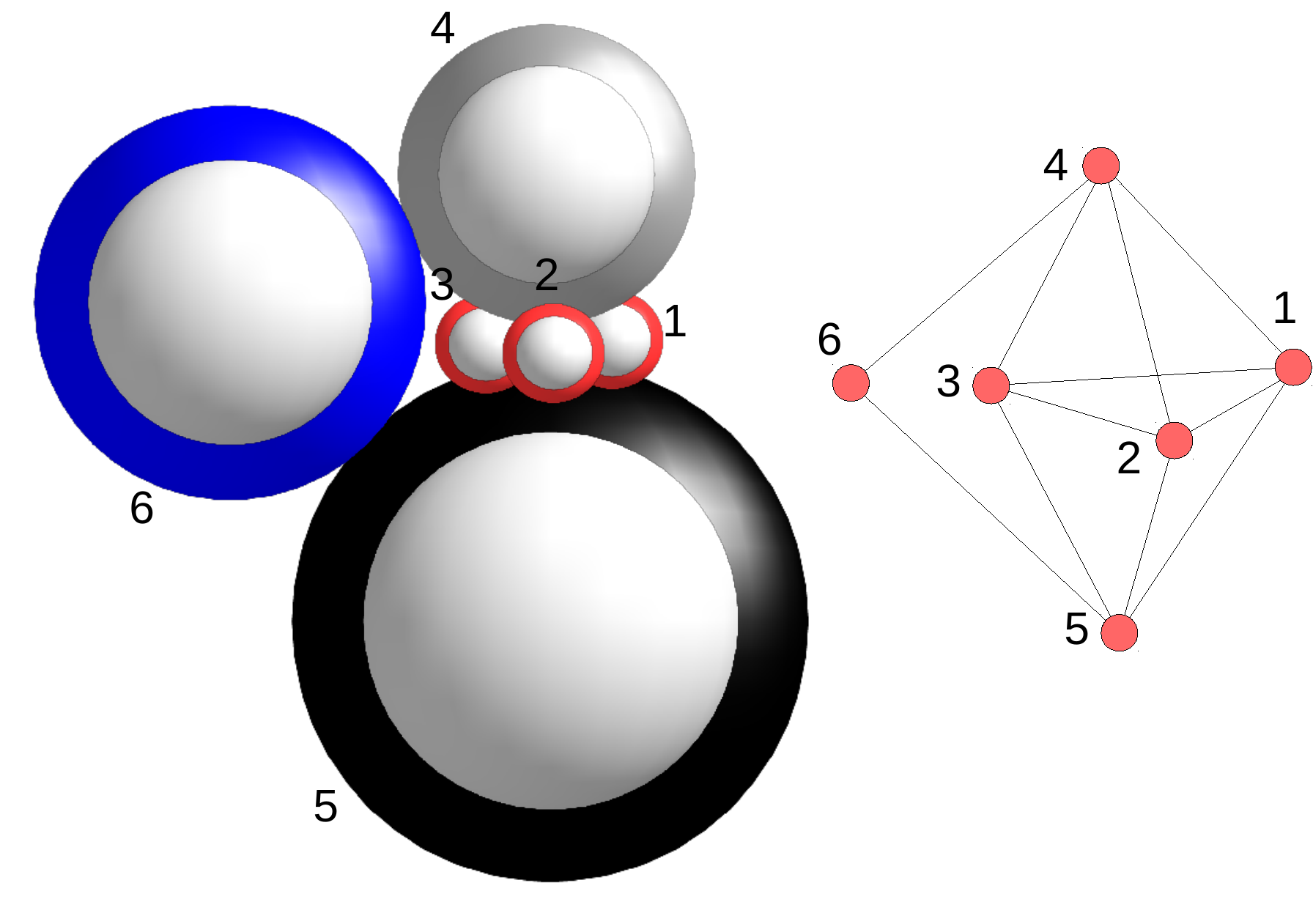}
 \caption{Any assembly configuration corresponding to the active constraint graph $G$ has its strict congruence group strictly contained in $\staba G$.
Here the bunches are singleton spheres and bunches of the same color have the same $\mathcal{C}$, $r$ and $\delta$.}
 \label{fig:config_graph_aut2}
 \end{figure}
 
 (2) $\staba G  = \staba R_G$:
from the definition of
permutations in the weak automorphism group of the assembly configuration
space, it follows that $\staba G \subseteq \staba R_G$.
 To show $\staba R_G \subseteq \staba )$,
 consider any element $\psi \in \staba R_G$. 
 For any assembly configuration $\assemblycfg \in R_G$, if a pair of spheres $(x,y)$ are ``touching'' (i.e.\ they yield an edge in the corresponding active constraint graph),
 it must be the case that $(\psi(x), \psi(y))$ are also ``touching'' in $\psi(\assemblycfg)$, since $\acg = G(\psi(\assemblycfg)) = G$. Similarly, $\psi$ must mapping ``non-touching'' pairs to ``non-touching'' pairs. Therefore $\psi \in \staba G$.
 \end{proof}
 
 \begin{remark}
We expect the strict order preserving isomorphism group and the strict permuted congruence group of
an assembly configuration $\assemblycfg$ to lie between the strict congruence group $\staba \assemblycfg$ and the automorphism group $\staba G$ of its active constraint graph.
However, the containment relationship between these two groups is not clear.
 \end{remark}

\subsection{Symmetries in stratification, assembly path and pathway}

A \emph{stratification $\strat$} of the assembly configuration space $\cfgspace$ is a partition of the space into strata $\mathcal{X}_i$ of $\cfgspace$
that form a filtration $\emptyset \subset \mathcal{X}_0 \subset \mathcal{X}_1 \subset \ldots \subset \mathcal{X}_m = \cfgspace$, $m = 6(n - 1)$.
Each $\mathcal{X}_i$ is a union of active constraint regions $R_G$, where the corresponding active constraint graph $G$ has $m-i$ independent edges, i.e.\ $m-i$ inequality constraints are active.
Each active constraint graph $G$  is itself part of at least one, and possibly many, hence $l$-indexed, nested chains of the form  $\emptyset \subset G^l_0 \subset G^l_1 \subset \ldots \subset G^l_{m-i} = G \subset \ldots \subset G^l_m$.

These induce corresponding reverse nested chains of active constraint 
regions $R_{G^l_j}$:
$\emptyset \subset R_{G^l_m} 
\subset R_{G^l_{m-1}} \subset \ldots \subset 
R_{G^l_{m-i}} = R_{G} \subset \ldots R_{G^l_0}$.
Note that here for 
all $l, j$, $R_{G^l_{m-j}} \subseteq \mathcal{X}_j$ is closed and 
$j$ dimensional.
See Figure~\ref{fig:roadmap} for an example of assembly configuration 
space stratification.

Given two active constraint graphs $G_i$ and $G_j$, 
$R_{G_i}$ 
(resp.\ $G_i$) is a 
\emph{parent} of $R_{G_j}$ (resp.\ $G_j$) (resp.\ $R_{G_j}$ is a 
\emph{child} of $R_{G_i}$) if $G_i \subsetneq G_j$ and there does not exists an 
active constraint graph $G_m$ such that $G_i \subsetneq G_m \subsetneq G_j$.
The parent-child relation provides a Hasse diagram of active constraint regions 
in the stratification of $\cfgspace$.

\begin{figure}[htbp]
	\centering
	\includegraphics[width=0.7\linewidth]{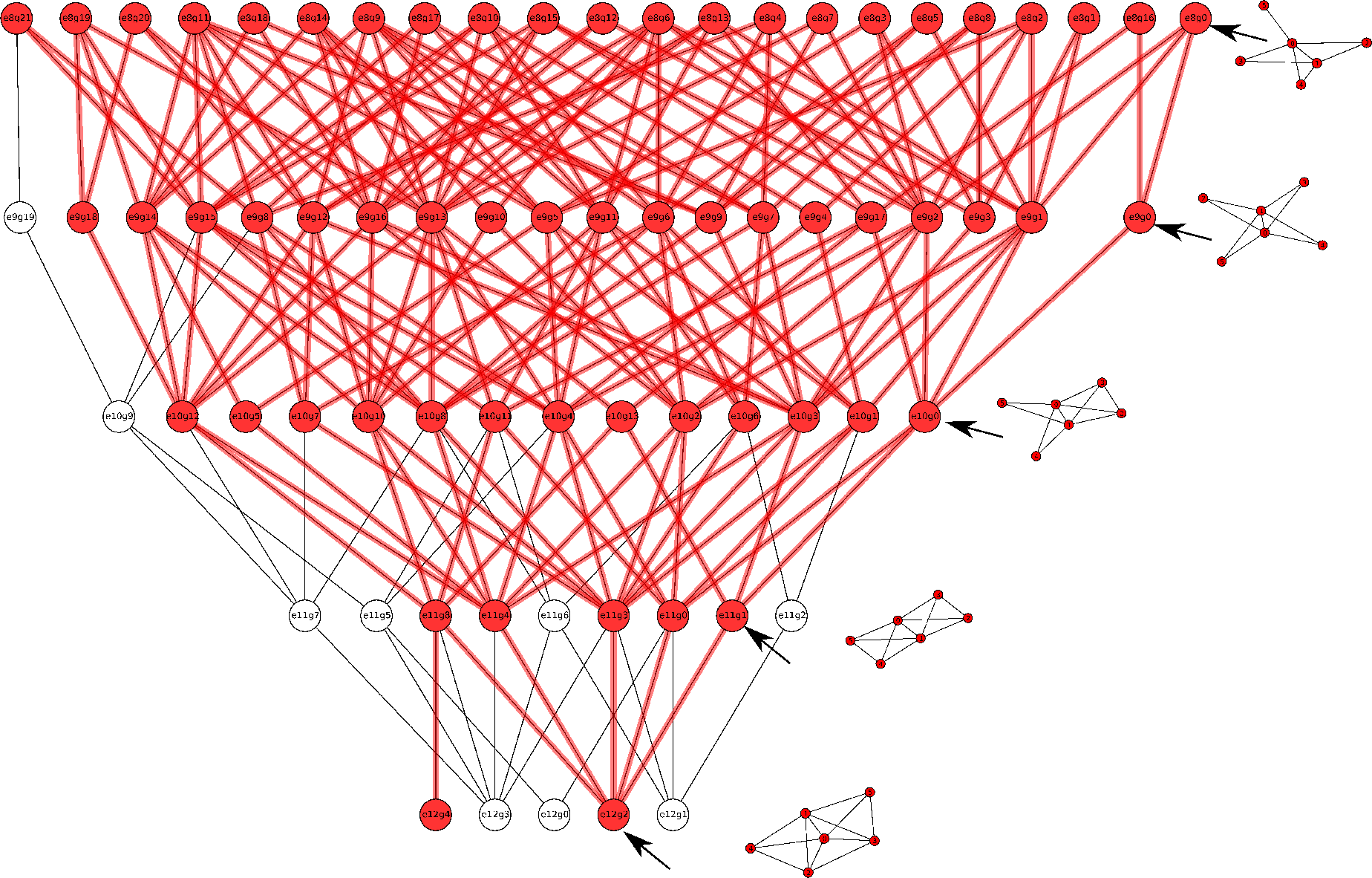}
	\caption{
		A fundamental region of the stratification for the assembly configuration space of the assembly configurations in Figure~\ref{fig:acg2} of 6 bunches, with each bunch being a singleton sphere and  all bunches identical. So $\auta$ is the complete symmetric group of permutations of 6 elements, $S_6$.
		Each node shown is an orbit representative of an active constraint region corresponding to an active constraint graph. 
		The grey part is those active constraint graphs (orbit representatives) whose corresponding constraint regions are empty.
		The example active constraint graph representatives on the right have arrows pointing to their regions in the stratification.
		The labels in the circles are unimportant: they are automatically generated and specify an orbit of an active constraint graph (example shown on right).
	}
	\label{fig:roadmap}
\end{figure}

An \emph{assembly path} from $G_1$ to $G_m$ in the stratification 
is a sequence 
$G_1 \subsetneq G_2 \subsetneq G_3 \subsetneq \ldots \subsetneq G_m$
 where $G_{i+1}$ is a child of $G_i$ for all $1 \le i \le m$.
A \emph{coarse assembly path} from $G_1$ to $G_m$ in the stratification  is a 
sequence
$G_1 \subsetneq G_2 \subsetneq G_3 \subsetneq \ldots \subsetneq G_m$ 
where $G^*_{i+1}$ has exactly one new rigid component $S$ not in $G^*_i$,
with 
$S$ containing a set of two or more rigid components $S_1 \ldots S_m$ of $G_i$.
In addition, for all proper subsets $Q \subsetneq \{S_1 \ldots S_m\}$ with 
$|Q| \ge 2$, 
the subgraphs of $G^*_{i+1}$ induced by $Q$ are not rigid. 
(The \emph{rigid components} of a graph are the maximal rigid subgraphs. 
Two rigid components cannot intersect on more than two vertices. 
We refer the reader to combinatorial rigidity concepts in~\cite{Graver}.)

For example,
In Figure~\ref{fig:roadmap}, 
the sequence of active constraint graphs on the right form an assembly path.


An \emph{assembly forest} 
corresponding to a coarse assembly path from  $G_1$ to $G_m$
 is the unique 
forest where the leaves are the maximal rigid components of $G_1^*$. 
The internal nodes are the new rigid components $S$ occurring in some 
$G_{i+1}^*$ in the path. 
 The children of $S$ are the set of rigid components 
$S_1 \ldots S_m$ contained in $S$ that occur in $G_i^*$.
 The \emph{roots} of the forest are the rigid components of $G_m^*$. 
An \emph{assembly tree} is an assembly forest with only one root. 
See Figure~\ref{trees4} in Section~\ref{sec:4} for examples of assembly trees~\cite{mvs2006,BSV,BoSi}.
%

A \emph{full (coarse) assembly path} is an (coarse) assembly path from 
$G_1$ to $G_m$,
 where $G_1$ is the empty active constraint graph, and $G_m^*$ 
is a rigid active constraint graph.
A \emph{(coarse) assembly path from primitives} has the first 
property of the full
assembly path, i.e.\ $G_1$ is the empty active constraint 
graph, but not the last property, 
i.e.\ $G_m$ can be any active constraint graph.
The \emph{full assembly tree} and \emph{assembly tree from 
primitives}
are also defined in this way.







A \emph{path} between full active constraint graphs $G$ and $H$ where 
$G \nsubseteq H$ and $H \nsubseteq G$ is a sequence $G = G_i, G_{i+1}, G_{i+2}, 
\ldots, G_{i+m}=H$,
where any pair $G_{i+k}$ and $G_{i+k+1}$ are on some assembly 
path, and $G_{i+k} \subsetneq G_{i+k+1}$ if $k$ is even, $G_{i+k} 
\supsetneq G_{i+k+1}$ if $k$ is odd.

The \emph{fundamental domain} of 
the stratification $\strat$
is the minimal 
sub stratification $\substrat$ such that
$\bigcup_{\pi \in \auta} 
\pi(\substrat) = \strat$,
where $\pi$ acts on $\substrat$ via its action on 
the active constraint regions (resp.\ active constraint graphs) of $\substrat$.
In other words the active constraint regions (resp. active constraint graphs) 
in $\substrat$ are
orbit representatives of active constraint regions (resp. 
active constraint graphs) under $\auta$.

An \emph{assembly pathway} is an orbit of an assembly tree under $\auta$. 
The definition extends to full, and coarse assembly trees.

\subsection{Example illustrating above symmetries}

Some of the symmetry concepts defined here were used in \cite{OPPS} to efficiently compute path and higher dimensional region intervals in sphere-based assembly configuration spaces more efficiently reproducing and extending the results in~\cite{Holmes-Cerfon2013}.
We give a brief description here in the form of an example:

\begin{example}
\label{eg:tree}

\begin{figure}[htbp]
	\centering
	\includegraphics[width=0.4\linewidth]
	{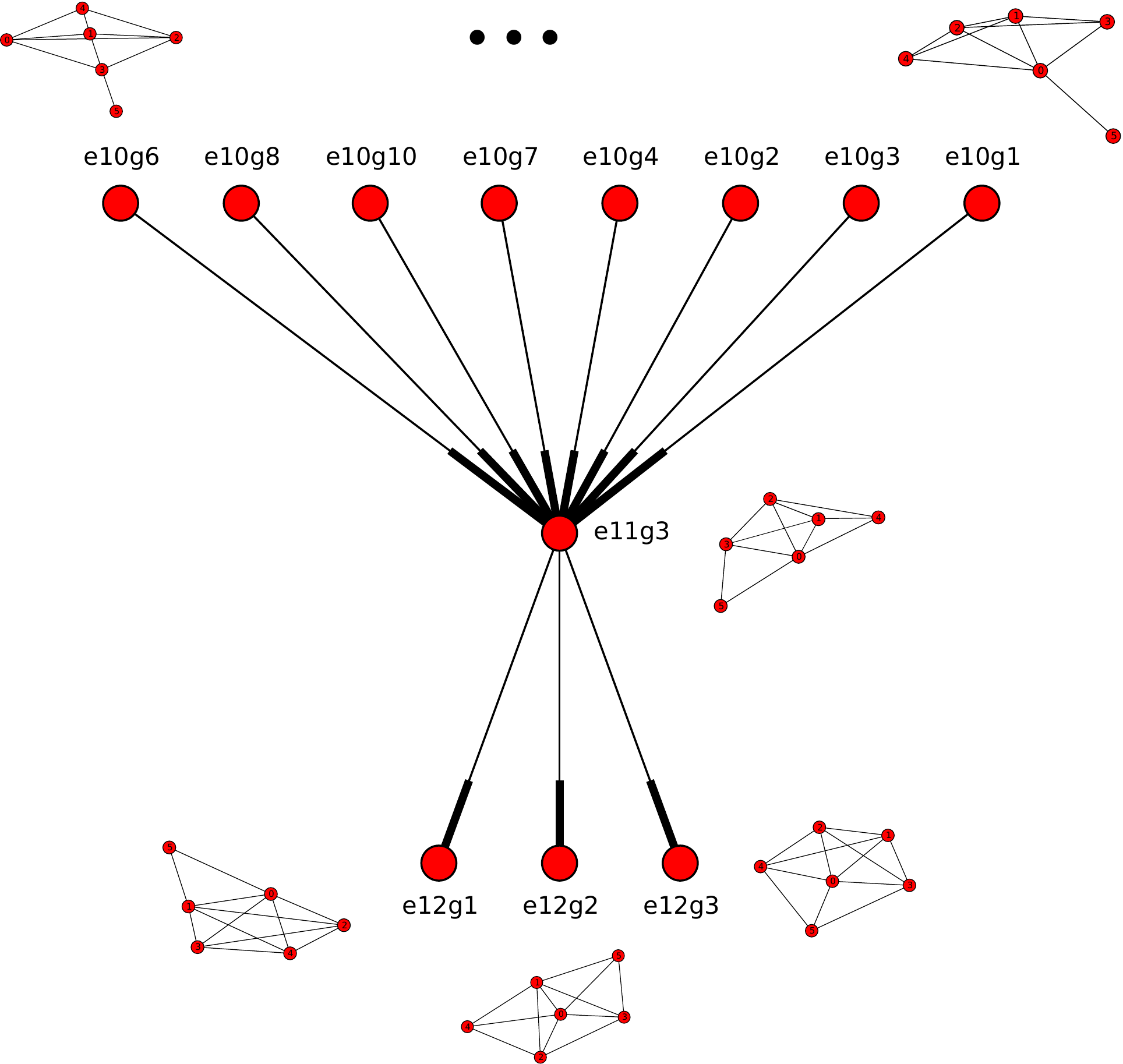}
	
	\caption{The neighbors of one active constraint graph in the Hasse diagram
		of the stratification for the assembly system in 
		Figure~\ref{fig:acg2}}
	\label{fig:hasse}
\end{figure}

As an example, 
Figure~\ref{fig:roadmap} 
shows the Hasse diagram of the fundamental region of a stratification of an assembly 
system of 6 bunches that are identical singleton spheres considered first in~\cite{Holmes-Cerfon2013}. 
Figure~\ref{fig:hasse} shows an (orbit representative of an) active constraint graph of the system 
together with its parents and children in the Hasse diagram.

In addition, orbit representatives of paths help in improving efficiency of path integrals.
in Figure~\ref{fig:roadmap}, any path that goes down from the top 
of the diagram to the bottom is the orbit representative of an assembly path.
In Figure~\ref{fig:hasse},
the 
sequence $e10q6 \subsetneq e11g3 \subsetneq e12g2$ is the orbit representative of an assembly path 
but not 
a coarse assembly path, as none of $e11g3$'s rigid components contains two or 
more rigid components of $e10g6$.
On the other hand, the sequence 
$e10q6 \subsetneq e12g2$ is the orbit representative of a  coarse assembly path.

\end{example}



%

%

\section{Enumerating Simple Assembly Pathways}
\label{sec:definition}
\label{sec:4}

In this section, 
we consider the action of the strict congruence group of a single final configuration on its assembly trees, and use generating functions to count the number and sizes of simplified assembly pathways \cite{BSV}.
Note that our approach could potentially be applied for all other groups defined in Section~\ref{sec:formal}, the 
largest of which is the weak automorphism group of the final
configuration, which would be the same as the  weak automorphism group of the assembly configuration space.
 
A simple assembly is modeled by a rooted tree,
the leaves are abstract representation of individual bunches, 
the root representing the 
final assembled configuration. 
The internal vertices 
represent intermediate stages of assembly, simplified to be subsets instead
of subgraphs of the root.
This simplification results in a loss of information about the assembly configuration space
and active constraint graphs of the intermediate stages of assembly. 
To compensate,  the group is taken to be the automorphism group
$G$
of the graph of the assembled structure at the root
instead of the weak automorphism group $\auta$ of the assembly configuration space.

The definitions of assembly tree and  pathway are simplified as follows.
Given a finite group $G$ acting on a finite set $X$, we will define
a {\it simplified assembly pathway} for the pair $(G,X)$.  First, a {\it simplified assembly tree} is 
a  rooted tree for which each internal vertex has at least two
children and whose leaves are bijectively labeled with elements of a
set $X$.  There is an induced labeling on all the vertices of a simplified  assembly tree
by labeling a vertex $v$ by the set of labels on 
the leaves that are descendents of $v$.  We identify each vertex of a
 simplified assembly tree with its label.  Two simplified assembly trees are considered identical if
there is a root preserving, adjacency preserving, and label preserving bijection
between their vertex sets.  The $26$ simplified assembly
trees with four leaves, labeled in the set $X = \{1,2,3,4\}$ are shown
in Figure~\ref{trees4}.

For a simplified assembly tree $\tau$,  the action of $G$ on $X$ induces a natural
action of $G$ on the power set of $X$ and thereby on the set of
vertices of $\tau$.  Let $\mathcal T_X$ denote the set of all simplified assembly trees
for $X$.    If $g \in G$, then define the
tree $g(\tau)$ as the unique simplified assembly tree whose set of vertex labels
(including the labels of internal vertices) is $\{g(v): v\in \tau\}$.
Thus we have an induced action of $G$ on $\mathcal{T}_X$.
Each orbit of this action of $G$ on ${\mathcal T}_X$ consists of
a set of simplified assembly trees called a {\it simplified assembly pathway for}
$(G,X)$.  

\begin{figure}[!ht]
\vskip 2mm
\begin{center}
\includegraphics[width=3in]{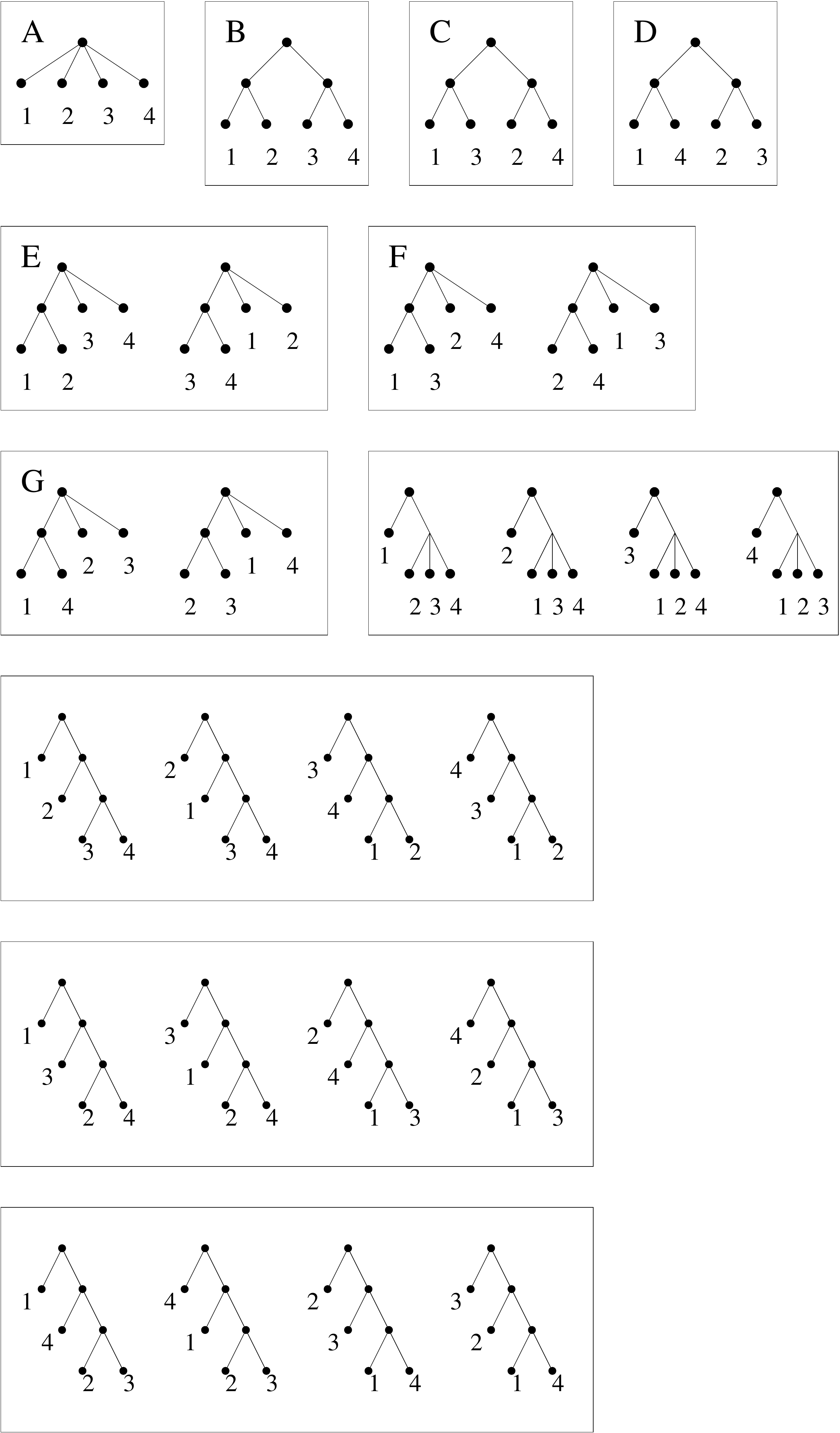}
\end{center}
\caption{Klein 4-group acting on $\mathcal{T}_4$.} 
\label {trees4}
\end{figure}

\begin{example}[Klein 4-group acting on $\mathcal{T}_4$]  \label{klein} 
Consider the Klein 4-group $G = \Z _2 \oplus \Z _2$ acting
on the set $X = \{1,2,3,4\}$. Writing $G$ as a group of 
permutations in cycle
notation, this action is
$$G = \{(1)(2)(3)(4), \, (1\;2)(3\;4), \, (1\;3)(2\;4), \,
(1\;4)(2\;3)\}.$$  For
this example there are exactly 11 simplified assembly pathways, which are
indicated in Figure~\ref{trees4} by boxes around the orbits.  There
are four simplified assembly pathways of size one, i.e., with one simplified assembly tree in the
orbit, three simplified assembly pathways of size two, and four simplified assembly pathways of size
four. 
\end{example}

For any subgroup $H$ of $G$, let $t_X(H)$ denote the number of trees in
 $\mathcal{T}_X$
that are fixed by every element of $H$. Furthermore,
 let $\overline{t}(H) := \overline{t_X}(H)$ denote
the number of trees in $\mathcal{T}_X$ that are fixed by every element of
$H$ but by no other elements of $G$. In other words,
\begin{equation}
\overline{t}_X(H) =
 |\{ \tau \in \mathcal{T}_X \, | \, stab_G(\tau) = H \}|. 
\label{orbitsize}
\end{equation}

The first theorem below reduces the enumeration of simplified assembly pathways to the
calculation of $\overline t(H)$ for subgroups $H$ of $G$. The index of a subgroup 
$H$ in $G$, i.e.  the number of left (equivalently, right), cosets of
$H$ in $G$ is denoted by $(G:H)$.  By Lagrange's Theorem, this index equals $|G|/|H|$. 
The second theorem below reduces the calculation of $\overline t(H)$ to the calculation of $t(H)$. 
The desired quantities  $\overline{t}_X(H)$ are computed from the numbers $t_{X}(H)$
 using M\"obius inversion on the lattice of subgroups of $G$.

\begin{theorem} \label{orbit}
The number of trees in any simplified assembly pathway for $(G,X)$ divides $|G|$.
If $m$ divides $|G|$, then the number $N(m)$ of simplified assembly pathways of
cardinality $m$ is
 $$N(m) = \frac{1}{m} \; \sum_{H\leq G \,: \, (G:H) = m} \;
\overline{t}(H).$$
\end{theorem}

\begin{theorem} \label{mobius} Let $G$ be a group acting on a set $X$.
If $H$ is a subgroup of $G$, then
$$\overline{t_X}(H) = \sum_{H\leq K \leq G} \; \mu (H,K)\; t_X(K),$$
where $\mu$ is the M\"obius function for the lattice of subgroups of $G$.
\end{theorem}

\begin{example}[Klein 4-group acting on ${\mathcal T}_4$ - continued]
Theorem~\ref{orbit}, applied to our previous example of $\Z _2 \oplus \Z _2$
acting simply on $\{1,2,3,4\}$, states that the size of a simplified assembly
pathway must be $1$, $2$ or $4$, since it must be 
a divisor of $4 = |\Z _2 \oplus \Z _2|$.
To find the number of pathways of each size, note that
$G$ has three subgroups of order 2, namely
$$\begin{aligned}
K_1 &= \{ \,(1)(2)(3)(4), (1\;2)(3\;4) \, \}, \\
K_2 &= \{ \,(1)(2)(3)(4), (1\;3)(2\;4) \, \}, \\
K_3 &= \{ \,(1)(2)(3)(4), (1\;4)(2\;3) \, \},
\end{aligned}$$
and that
$$\begin{aligned}
& \overline{t}(G) = 4, \\
& \overline{t}(K_1) = \overline{t}(K_2) =\overline{t}(K_3) = 2, \\
& \overline{t}(K_0) = 16,
\end{aligned}
$$
where $K_0$ denotes the trivial subgroup of order 1.
The simplified assembly trees in ${\mathcal T}_X$ that are fixed by all elements
of $G$ are shown in Figure~\ref{trees4}, $A, B, C, D$.  For $i =
1,2,3$, those simplified assembly trees in ${\mathcal T}_X$ that are fixed by all
elements of $K_i$ and by no other elements of $G$ are are shown in
Figure~\ref{trees4}, $E, F, G$, respectively. The remaining 16 simplified assembly
trees in Figure~\ref{trees4} are fixed by no elements of $G$ except
the identity.  Therefore, according to Theorem~\ref{orbit},
the number of pathways of size 1, 2 and 4 are, respectively,
\[\begin{aligned}
\overline{t}(G) &= 4, \\
\frac12 \, \left( \, \overline{t}(K_1) + \overline{t}(K_2) +
\overline{t}(K_3)\, \right) = \frac12 \, (2+2+2) &= 3, \\
\frac14 \, \overline{t}(K_0) &= 4.
\end{aligned}
\]
\end{example} 

The problem of enumerating simplified assembly pathways is reduced, using Theorems~\ref{orbit} and \ref{mobius},
to calculating the number $t(G)$ of simplified assembly trees fixed by a given group $G$.  This is done using permutation
group theory and generating functions.  It will be assumed, as is the case in many of the biological appllications, that $G$ acts 
freely on $X$, i.e., if $g(x) = x$ for some $x\in X$, then $g$ must be the identity.    In this case
\[|X| := |X_n| = n \cdot |G|,\] 
where $n$ is the number of $G$-orbits in its action on $X$.    Denote by $t_n(G)$ the number of trees
in ${\mathcal T}_n := {\mathcal T}_{X_n}$ that are fixed by $G$.  We define the exponential generating function
\[ f_G(x) := \sum_{n\geq 1} t_n(G) \,\frac{x^n}{n!} \] for the sequence
$\{t_n(G)\}$.  

If $G$ is the trivial group of order one, then let us denote this
generating function simply by $f(x)$. This is the generating function
for the total number of rooted, labeled trees with $n$ leaves in which
every non-leaf vertex has at least two children. For $H\leq G$, let
$$\widehat f_H (x) = 
\frac{1}{(G:H)}\; f_H\left( (G:H)x \right).$$
\begin{theorem} \label{generating} The generating function
$f_G(x)$ satisfies the following functional equations:
$$1-x+2f(x) = \exp \, (f(x)),$$ and for $|G|>1$, $$1+2 f_G(x) =
\exp\, \left(\sum_{H \leq G} \; \widehat f_H (x)\right).$$
\end{theorem}

Althogh proofs are omitted in this survey, the rather involved  proof of Theorem~\ref{generating} 
relies on, in addition to generating function techniques, a characterization
of block systems arising from a group acting on a set and 
 a recursive procedure for constructing all trees in $\mathcal T_X$ that are fixed by $G$. 
(See \cite[Theorems 9 and 14]{BSV}.)

\begin{remark}  Finding the generating function $f_G(x)$ depends on first finding the
generating functions $f_H(x)$ for proper subgroups $H$ of $G$.  In that sense, the
procedure for finding $f_G(x)$ is recursive, proceeding up the lattice of
subgroups of $G$, starting from the trivial subgroup.  

It is also worth mentioning that subgroups that are conjugate in $G$ 
have the same generating function.
\end{remark}

\begin{example}[Klein 4-group acting on ${\mathcal T}_4$ - continued]

Consider $G = \Z_2 \oplus \Z_2$ acting on $X_n$. Recall that $|X_n| =
4n$, the integer $n$ being the number of $G$-orbits.  Recall that the subgroups of $G$ are
$K_0, K_1,K_2,K_3, G$, where $K_0$ is the trivial
group and
\[\begin{aligned}
K_1 &= \{ \, (1)(2)(3)(4), (1\;2)(3\;4) \, \}, \\
K_2 &= \{ \, (1)(2)(3)(4), (1\;3)(2\;4) \, \}, \\
K_3 &= \{ \, (1)(2)(3)(4), (1\;4)(2\;3) \, \}.
\end{aligned}\]
The functional equations in the statement of Theorem~\ref{generating}
are
$$\begin{aligned}
1-x+2f(x) &= \,\exp \, (f(x)) \\
1+2f_{K_i}(x) &=\, \exp \, \left(\frac12 \,f(2x)+ f_{K_i}(x) \right) 
\quad  \text{for $i=1,2,3$, and}\\
1+2f_G (x) &= \,\exp \,  \left(\frac14 \, f(4x) +\frac12 \, f_{K_1}(2x)+
\frac12 \, f_{K_2}(2x)+\frac12 \, f_{K_3}(2x)+ f_G (x) \right). 
\end{aligned}$$
Using these equations and MAPLE software, the coefficients of the
respective generating functions provide the following first few values
for the number of fixed simplified assembly trees. For the first entry $t_1(G)
=4$ for the group $G$, the four fixed trees are shown in
Figure~\ref{trees4} $A$, $B$, $C$, $D$.  For trees with eight leaves there are
$t_2(G) =104$ simplified assembly trees fixed by $G = \Z_2 \oplus \Z_2$, and so
on.

$$\begin{aligned}
t_n(K_0) \quad &: \quad  1, 1 , 4, 26, 236, 2752  \\
t_n(K_i) \quad &: \quad  1, 6, 72, 1312,  32128, 989696 \\
t_n(G) \quad &: \quad 4, 104, 4896, 341120, 31945728, 3790876672.
\end{aligned}$$
\end{example}

\begin{example}[The icosahedral group acting on a viral capsid]

 A {\it symmetry} of a polyhedron  is a transformation in {\em SE}$(3)$
that keeps the polyhedron, as a whole, fixed, and a {\it direct
symmetry} is similarly defined.  The {\it
icosahedral group} is the group of direct symmetries
of the icosahedron. It is a group of order 60 denoted $G_{60}$.  

A viral capsid  assembly configuration is modeled by a polyhedron $P$
with icosahedral symmetry. Its set $X$ of facets represent the
protein monomers. The icosahedral group acts on $P$ and hence on the
set $X$.  It follows from the so-called quasi-equivalence theory of the capsid
structure  that $G_{60}$ acts freely on $X$.  We have $|X| := |X_n| = 60n$, 
where $n$ is the
number of orbits in the action of the icosahedral group on $X$.  
Not every $n$ is possible for a viral capsid; $n$
must be a $T$-number, that is, a number of form $h^2+hk+k^2$, 
where $h$ and $k$ are nonnegative integers. 

\bigskip\noindent
{\em Note.}  An icosahedral viral capsid assemly configuration has a corresponding  icosahedral active constraint graph. And the group $G_{60}$, viewed as a subgroup of the symmetric group $S_{60}$ is the automorphism group of this active constraint graph. 
As mentioned in the beginning of this section, we are interested in the orbits of simplified assembly trees under the action of this automorphism group.   However, we continue to use the more intuitive view of $G_{60}$ as a geometric group. 

\medskip\noindent
Before the
number of simplified assembly trees can be enumerated, 
basic information about the icosahedral group is needed. The group $G_{60}$ consists of:

\begin{itemize}
\item the identity,
\item 15 rotations of order 2 about axes that pass through the
midpoints of pairs of diametrically opposite edges of $P$,
\item 20
rotations of order 3 about axes that pass through the centers of
diametrically opposite triangular faces, and
\item 24
rotations of order 5 about axes that pass through diametrically
opposite vertices.  
\end{itemize}

There are 59 subgroups of $G_{60}$ that play a
crucial role in the theory.  Besides the two trivial subgroups, they
are the following:

\begin{itemize}
\item 15 subgroups of order 2, each generated by one of the
rotations of order 2,
\item 10 subgroups of order 3, each generated by one of the
rotations of order 3,
\item 5 subgroups of order 4, each generated by rotations of order
2 about perpendicular axes,
\item 6 subgroups of order 5, each generated by one of the
rotations of order 5,
\item 10 subgroups of order 6, each generated by a rotation of
order 3 about an axis L and a rotation of order 2 that reverses
L,
\item 6 subgroups of order 10, each generated by a rotation of
order 5 about an axis L and a rotation of order 2 that reverses
L,
\item 5 subgroups of order 12, each the symmetry group of a
regular tetrahedron inscribed in $P$.
\end{itemize}

\noindent From the above geometric description of the subgroups, it
follows that all subgroups of a given order are conjugate in the group
$G_{60}$.  Representatives of the conjugacy classes of the subgroups
of the icosahedral group are denoted by $G_0, G_2, G_3, G_5, G_6,
G_{10}, G_{12}, G_{60}$, where the subscript is the order of the
group. The set of subgroups of $G_{60}$ forms a lattice, ordered by
inclusion.  A partial Hasse diagram for this lattice $\bf L$ is shown
in Figure~\ref{Hasse}.  The number on the edge joining $G_i$ (below) and
$G_j$ (above) indicate the number of distinct subgroups of order $i$
contained in each subgroup of order $j$.  The number in parentheses on
the edge joining $G_i$ (below) and $G_j$ (above) indicate the number
of distinct subgroups of order $j$ containing each subgroup of order
$i$.  The M\"obius function of $\bf L$ is shown in Table 1. The
entry in the table corresponding to the row labeled $G_i$ and column
$G_j$ is $\mu(G_i,G_j)$.

\begin{figure}[htb]
\vskip 2mm
\begin{center}
\includegraphics[width=2in]{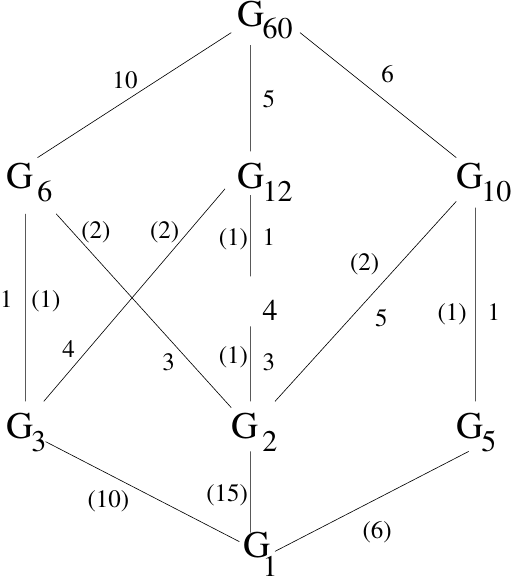}
\end{center}
\caption{Partial Hasse diagram for the lattice of subgroups of the
icosahedral group.} 
\label {Hasse}
\end{figure}

\begin{table}[htb] 
 \begin{center}
\includegraphics[width=2in]{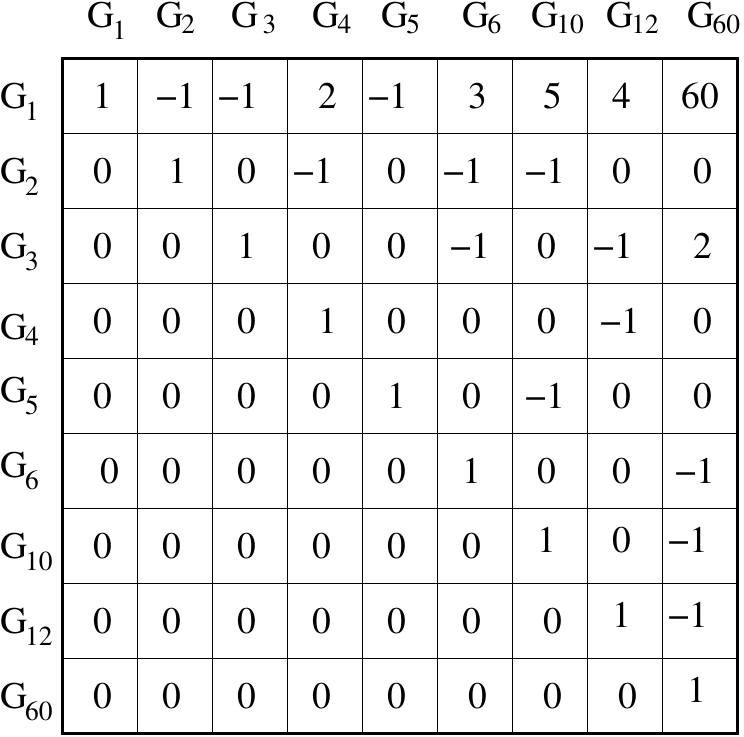}
\vskip 3mm
\caption{The values of the M\"obius function of the subgroup lattice
of $G_{60}$.}
 \end{center}
 \label{mob}
\end{table}

Consider the case $|X| = 60$, i.e., for the $T=1$ capsid.  
Using Theorem~\ref{generating} and MAPLE software, the generating functions
$f_{G_{i}}(x)$ were computed, and hence their coefficients 
$t_{60/i}(G_{i})$ which
count  simplified assembly trees that are 
fixed by any copy of $G_{i}$ were also computed.
Note that, since $|X|=60$, the number of orbits of $G_i$ in its action on $X$
is $60/i$. Substituting these values into Theorem~\ref{mobius} and using the 
M\"obius Table 1 yields
the numerical values for $\overline{t}_{60/i}(G_{i})$, the 
number of simplified assembly trees over $X$ with $|X| = 60$ that are fixed by $G_i$
but by no other elements of $G_{60}$. In other words, these
are the numbers of  trees whose stabilizer in $G_{60}$
is exactly $G_i$.
Substituting these numbers $\overline t$ into Theorem~\ref{orbit}, we arrive at the number of simplified assembly 
pathways of each possible size:

\[\begin{array}{ll}
204   & \text{simplified assembly pathways of size $1$} \\
\sim \num{168e8} \quad &\text{simplified assembly pathways of size $5$} \\
\sim \num{223e9}  &\text{simplified assembly pathways of size $6$}\\
\sim \num{613e17}  &\text{simplified assembly pathways of  size $10$}\\
\sim \num{102e17}  &\text{simplified assembly pathways of size $12$}\\
\sim \num{334e28}  &\text{simplified assembly pathways   of size $15$} \\
\sim \num{504e31}  & \text{simplified assembly pathways of size $20$}\\
\sim \num{835e51}  &\text{simplified assembly pathways of size $30$} \\ 
\sim \num{320e99} &\text{simplified assembly pathways of size $60$}
\end{array}\]

\end{example}

\section{Open Questions}
\label{sec:5}

\subsection{Enumeration problems in (non-simplified) assembly framework}

We are interested in the following enumeration problems related to the action of $\auta$
for the framework in Section~\ref{sec:formal}:

\begin{enumerate}

\item How to compute the size of orbits\slash{}stabilizers and the number of orbits under $\auta$
for 
assembly configurations, active constraint graphs, active constraint regions, (coarse) assembly paths and assembly trees\slash{}forests?

\item How to compute the number of coarse assembly paths that correspond to a particular
 assembly tree\slash{}forest?

\item Given two active constraint graphs $G$ and $H$, where $G$ and $H$ are incomparable, 
i.e.\ $G \nsubseteq H$ and $H \nsubseteq G$,
how to compute the number of paths between them?

\item Given two active constraint graphs $G_1$ and $G_m$, where $G_1 \subsetneq G_m$,
how to compute the number of (coarse) assembly paths from $G_1$ to $G_m$?

\item What are the orbits of the (coarse) assembly paths in (4) under the action of $\stab_\auta(G_m)$?

\item What are the orbits of the (coarse) assembly paths in (4) under the action of 
the group $H$, where $H = \auta$ if  $\stab_\auta(G_1) = \auta$ (i.e.\ $G_1$ is the empty active constraint graph), or $H = \stab_\auta(G_1)\cap \stab_\auta(G_m)$ otherwise?
\end{enumerate}

\subsection{Symmetries within an active constraint region via Cayley configurations}

So far, we have discussed
the orbit of an active constraint region and active constraint graph, and
pointed out that it is sufficient to deal with a single orbit
representative provided we are able to compute the multiplying factors
associated with the size of the orbit, stabilizer, number of orbits etc.

In fact,  a single active constraint region  could be
decomposed into the union of nontrivial subregions that form the orbit of
a fundamental region, leading to enormous efficiencies in sampling, computation of volumes that are currently hoplessly intractable in high dimensional configuration spaces
as discussed in the Introduction. 

In fact since the fundamental region itself could have subregions with varying 
orders of stabilizers, we could decompose into more than one orbit representative,
with different stabilizers. In any
case, sampling or computing the volume of an active constraint region is
simplified by sampling these  fundamental subregions and computing the
size of their orbits.

One way to obtain such a decomposition  of an active constraint region $R_G$
is via the locally complete Cayley (assembly) configurations $\delta_F$ corresponding to the active constraint graph $G$. 
Convex Cayley configuration spaces highlight the 
key difference between assembly and other constraint systems
e.g., folding. This difference is captured in the combinatorial structure  of 
active
constraint graphs.  A Cayley parameter for an active constraint region $R_G$ is a non-edge of its active constraint graph $G$.
For specific sets of non-edges $F$, the set of vectors
$\lambda_F$ of  attainable lengths of $F$ -
(in 3D realizations of a linkage
$(G,\delta)$ with
underlying graph $G$ and edge lengths $\delta$)  -  is always
convex for any given lengths $\delta$  (that is, for
all the 3D realizations of the bar-joint constraint system or linkage
$(G,\delta)$). This set is called the (3-dimensional)
\dfn{Cayley configuration space} of the linkage $(G,\delta)$ on the
Cayley parameters $F$, denoted $\Phi_F(G,\delta)$ and can be viewed as
a ``projection'' of the space of pairwise distance vectors of realizations of $(G,\delta)$ on
the Cayley parameters $F$. Such graphs $G$ are said to have
\dfn{convexifiable Cayley configuration spaces with parameters $F$}
(short: \dfn{convexifiable}). Convexity permits the use of convex programming techniques for improving efficiency of sampling, search, volume computations etc. for the configuration space.

The concept is best explained 
using  key theorems of the first author in  
\cite{sitharam2010convex,sitharam2014flatten} discussed in Section \ref{sec:5}.

We assume knowledge of common graph operations such as $k$-sums and
resulting partial $k$-trees, a minor-closed class 
(partial 2-trees are  series-parallel graphs with a forbidden minor $K_4$). 
\begin{theorem}
    \cite{sitharam2010convex} A graph $H$ has a convexifiable Cayley configuration space  with
    parameters $F$ if and only if for each $f\in F$  all the minimal
    2-sum components of $H\cup F$ that contain both endpoints of $f$
    are partial 2-trees. The Cayley configuration space
    $\Phi_F(H,\delta)$ of a bar-joint system or linkage $(H,\delta)$
    is a convex polytope. When $H\cup F$ is a 2-tree, the bounding
    hyperplanes of this polytope are triangle inequalities relating
    the lengths of edges of the triangles in $H\cup F$.
\end{theorem}
    {\sl Note:} A major advantage of the convex Cayley method is that
    sampling the configuration space can be effected by standard methods of
    convex programming. Another  advantage is that the method is 
completely unaffected when $\delta$ are intervals rather than exact
values~~\cite{sitharam2010convex}.
\noindent
A different characterization of inherent Cayley convexity for a graph $G$ on
a set $F$ of non-edges as in the above section
has been proven also for higher dimensions $d$  \cite{sitharam2010convex}, 
\cite{cheng}, showing equivalence to a minor-closed property of $d$-flattenability introduced in 
\cite{BeCo:2008} and also for other, non-euclidean distances (norms) in 
\cite{sitharam2014flatten}.  Any realization of $H$ in a normed space can be flattened
into $d$-dimensional normed space (in the same norm)  maintaining the same edge distances.
\begin{theorem}
\cite{sitharam2014flatten}
A graph
$H$ is $d$-flattenable if and only if for every partition of $H$ into
    $G\cup
    F$,
    $G$ has a convex Cayley configuration space on $F$ in $d$-dimensions.
\end{theorem}

\bigskip\noindent
  
\subsubsection{Fundamental regions of Active constraint regions}

After  $G$ has been completed with the convexifying Cayley parameters
$F$, the locally rigid graph $G\cup F$   typically loses  symmetries present in $G$, i.e, the
automorphism group is smaller. However, $F$ can be replaced by any set
of edges $\pi(F)$ for $\pi \in \stab_\auta(G)$. 
Each locally complete Cayley configuration in the active constraint region $G$ is of the form $\delta_F$ (lengths of edges in $F$, where $G\cup F$ is rigid).
Each  cartesian (assembly) configuration within an active constraint region with graph $G$ corresponds bijectively to a globally complete Cayley configuration $(\delta_F,\delta_H)$ where $G\cup F$ is rigid and $G\cup F \cup H$ is globally rigid (or even $G\cup F \cup H$ is complete graph).

Thus when sampling the Cayley configuration space on $F$,  one can find  the boundaries of the
fundamental regions corresponding to the corresponding cartesian assembly configurations as follows.
For a Cayley configuration $\delta_F$,  all its generically finitely many real/cartesian configurations can be obtained as various corresponding values of $\delta_H$, which include the values
of $\delta_{\pi(F)}$.  
The boundary of a fundamental region occurs during sampling when we encounter a cartesian (assembly) configuration $c$ where the lengths of $\pi(F)$ correspond to already sampled lengths of $F$. 

Note that there could be a different decomposition into fundamental regions, corresponding to each cartesian configuration (type) corresponding to the Cayley configuration.  For example, for a different configuration $c'$ from the configuration $c$ above, the lengths of $\pi(F)$ may not correspond to already sampled lengths of $F$. 
Or,  there could be another element $\sigma \in \stab_\auta(G)$, with $\sigma\ne\pi$ where the lengths of $\sigma(F)$  in $c'$ could correspond to already sampled lengths of $F$.
In this manner, one can, in principle, algorithmically bound
fundamental regions $R^i_G$ of the active constraint region $R_G$,  by  inspecting the assembly configurations  corresponding to the Cayley configuration space on $F$, such
that the active constraint region $R_G$ is  the union of the orbits of the regions
$R^i_G$  (under the action of $\stab_\auta(G)$).

Efficiently finding these fundamental regions as well as the number and sizes of their orbits
is an open question, whose answer would enormously reduce the complexity of configurational entropy computations for assembly.

%

\subsection{$g$-unfixable unlabeled trees}

Call a tree $g$-{\it unfixable} if there is no
leaf-labeling so that the resulting labeled tree is fixed by the
permutation $g$, and let us say that a tree is $G$-{\it unfixable} if
it is $g$-unfixable for every nontrivial element of the group $G$.
A study of unlabeled trees that are $g$-unfixable may lead to relevant
related results. 
These properties are interesting for at least two reasons.  First,
they clarify the minimum quantifiable information in a labeled tree
that is necessary to decide if it is fixed by a group element $g$: if
the underlying unlabeled tree is $g$-unfixable, then the information
in the labeling is unnecessary to make this decision.  This may lead
to efficient algorithms that use properties of the automorphism group
of the tree to help in deciding whether a given labeled tree is fixed
by the given group.

\subsection{Depth of an assembly pathway}

A result of \cite{BoSi}  tells us that the orbit size of an assembly pathway is at least the depth of the
pathway. 
%
The number of assembly pathways and orbit sizes of assembly trees that
constitute a pathway, must be taken into consideration in defining any
probability space over pathways. If 
the  dynamics of transitioning between states along a pathway and thereby
the density of states influencing the configurational integral computation
 \cite{wales5} and other such factors nullify the vast differences in
symmetry-induced numeracy factors between pathways,
then that argument is yet to be made.
The local rules theories using simple geometric rules, ODEs and other
first principles physics based  simulations of assembly of viral capsids  
\cite {schwartz,bib:Berger,berger1,berger2,berger3,reddy,zlotnick2,marzec-day,johnson-assembly,johnson-assembly2,keef2006master} have been used to obtain the
assembly kinetics including rates and concentrations of intermediates, and
implicitly provide a probability distribution over pathways. A cautionary
note in \cite{misra2008pathway} uses an ODE based model of reaction kinetics to
question simplistic models of assembly pathways.  However, the model does
not contradict the simple and transparent thesis that when symmetric
structures form from identical units, the simple numeracy of orbit sizes
of assembly trees must be taken into consideration in any theory
predicting of likely assembly pathways. This paper shows the rich
intricacy of possible symmetries at play.
We in fact conjecture that this symmetry factor increases with
the depth of the pathway.  Proving this conjecture would strengthen the motivation for
studying the symmetry factor.

\subsection{Other questions}

Theorems \ref{orbit} and \ref{mobius} as well as our successful computation effort in the special case of $|X|=60$ and $T=1$ can serve as a motivation to revisit the following questions, first raised in \cite{BoSi}. 
\begin{enumerate}
\item
Given two symmetry invariant properties, 
how to compute the ratio of the  number of pathways that satisfy both of these properties 
to the number of symmetry classes that satisfy only one of these properties?

\item
What can we say about larger (icosahedrally) symmetric polyhedral graphs (larger $T$ 
numbers of viral capsids, for example), fullerenes and fulleroids and polyhedra with different symmetry groups?
In such cases, the computations of Section~\ref{sec:definition} 
can also be phrased as algorithmic questions,
where asymptotic complexity of the algorithm is expressed 
in terms of the number of facets of the polyhedron (or the $T$ number).

%
%

\item
To fully extend the techniques in  Section~\ref{sec:definition}  to the framework of Section~\ref{sec:formal},
each subassembly must be a rigid subgraph of the graph at the root.  Some
assembly trees fail to satisfy the rigidity condition and can never
occur (probability 0).  Such assembly trees are {\it geometrically
invalid}.  In addition, a valid assembly tree can be assigned a non-zero probability
according to how difficult it is to find a solution to the constraints
on each subassembly.  Computing this probability - called the {\it geometric
stability factor} - is necessary to make the required predictions.

Dropping the rigidity requirement, but maintaining the subgraph
(connectivity) requirement, in \cite{BoVi}, two of the authors study the number of assembly trees of graphs on {\em labeled} vertices. In that model, each graph has a trivial automorphism group, but the enumeration of assembly trees still leads to the use of a recent and very powerful technique from the theory of $D$-finite power series in several variables.

Incorporating a nontrivial automorphism group of the graph could help
understand the role of capsid symmetry in the RNA assembly model of
\cite{stockley2013packaging}, which purports that RNA viruses assemble by attaching to
the internal (symmetry breaking) genome strand since that would avoid
having to deal with the prohibitive number of possible assembly pathways.
It should be noted that in our precise and formal theory of assembly trees
and their orbits (our pathways), assembly has an underlying partial order
of stable intermediates, that are influenced by the connectivity and
rigidity, they are subgraphs of the underlying polyhedral graph given by
active constraints.  The informal definition of pathway in \cite{stockley2013packaging}
is a linear order (in our language, an assembly tree that is a path) given
by a hamiltonian circuit in the viral polyhedral (dual) graph. We are not
aware of a  clarification of  why the interactions  of a given monomer in
the sequence to multiple other monomers besides the previous one in the
sequence would be insignificant. If not, the assembly tree would indeed be
a partial order as in our case, and the tree would have a minimum fan-in 
required for rigidity, reducing the number of assembly trees significantly and 
reducing the number of their symmetry classes or orbits further,
whereby
this number alone is not a significant reason to adopt a alternate model
of assembly (such as RNA strand attachment) that cuts down the possible
pathways.
\end{enumerate}

As future work, we also aim to apply the symmetry framework developed in this paper to explain more experimental and theoretical results from previous literature.

\section{Conclusions}

In this paper, we developed a novel framework for symmetry in assembly under short range potentials and considered the symmetry groups of various objects studied in previous literature on assembly, including assembly configuration spaces, active constraint graphs, active constraint regions, assembly trees and pathways. 
The new Theorem~\ref{thm:containment} which formalizes the containment relations between stabilizer subgroups of  active constraint graph and corresponding assembly configurations.
We then demonstrated the new symmetry concepts to compute the sizes and numbers of orbits in two example settings appearing in previous work. 
The methods  can improve efficiency for large systems with multiple identical bunches and spheres  that have large order symmetry groups.
The new symmetry framework helps formalize a number of questions for future work.

\section{Acknowledgment}
 
 We thank Rahul Prabhu for his feedback and assistance in paper preparation.

%


\bibliographystyle{plain}
\bibliography{symmetry,nigms,jorg,Dmay04,draftnew}

\end{document}